\documentclass[11pt,a4paper,twoside]{article}
\usepackage{amsthm,amsfonts,amsmath}
\usepackage{xcolor}
\usepackage{enumitem}
\newtheorem{Theorem}{Theorem}%[section]
\newtheorem{Corollary}{Corollary}%[section]
\newtheorem{Definition}{Definition}%[section]
\newtheorem{Example}{Example}%[section]
\newtheorem{Lemma}{Lemma}%[section]
\newtheorem{Proposition}{Proposition}%[section]
\newtheorem{Remark}{Remark}%[section]

\author{
 Sourav Nayak\footnote{Department of Mathematics, Indian Institute of Technology - Hyderabad, Sangareddy-502285, India
 \newline e-mail: {\tt ma22resch11004@iith.ac.in},
 Orcid: 0009-0003-4330-8283}\ ,
 Dhriti Sundar Patra\footnote{Department of Mathematics, Indian Institute of Technology - Hyderabad, Sangareddy-502285, India
 \newline e-mail: {\tt dhriti@math.iith.ac.in, \tt dhritimath@gmail.com},
 Orcid: 0000-0002-7958-1690}
 \ and
 Vladimir Rovenski \footnote{Department of Mathematics, University of Haifa, Mount Carmel, 3498838 Haifa, Israel
 \newline e-mail: {\tt vrovenski@univ.haifa.ac.il},
 Orcid: 0000-0003-0591-8307.
 ({corresponding author})
 }
 }

\title{On the splitting of weak nearly ${\cal C}$-manifolds}

\begin{document}

\date{}

\maketitle

\begin{abstract}
The interest of mathematicians in metric $f$-manifolds, in particular, almost contact metric manifolds, is motivated by the study of the geometry and dynamics of contact foliations, as well as their applications in physics. Weak metric $f$-manifolds, defined by V. Rovenski and R. Wolak (2022), open a new perspective on classical theory of $f$-manifolds and discover new applications. In this paper, we study manifolds of this type, called  weak nearly ${\cal C}$-manifolds, which generalize almost ${\cal C}$-manifolds. We find conditions under which a $(2n+s)$-dimensional weak nearly ${\cal C}$-manifold becomes locally a Riemannian product, and characterize $(4+s)$-dimensional weak nearly ${\cal C}$-manifolds. The consequences of these theorems present new results for nearly ${\cal C}$-manifolds.

\vskip1.5mm\noindent
\textbf{Keywords}:
weak nearly ${\cal C}$-manifold, weak nearly K\"{a}hler manifold, curvature tensor, totally geodesic foliation, Killing vector field.
%, Riemannian product.

\vskip1.5mm
\noindent
\textbf{Mathematics Subject Classifications (2010)} 53C15, 53C25, 53D15
\end{abstract}

%\setcounter{secnumdepth}{4}
%%%%%%%%%%%%%%%%%%%%%%%%%%%%%%%%%%%%%%%%%%
%\setcounter{page}{1}

\section{Introduction}
\label{sec:00-ns}

Contact geometry has garnered increasing interest of mathematicians due to its significant role in physics, 
e.g. geometrical optics, mechanics, thermodynamics, geometric quantization, 
%integrable systems and 
control theory, general relativity, e.g.,~\cite{IZ1,K-2013}.
 An important subclass of almost contact metric manifolds $(M^{2n+1},{f},\xi,\eta,g)$, called nearly cosymplectic manifolds, is defined in \cite{blair1974} by condition that the symmetric part of $\nabla{f}$ vanishes.
These odd-dimensional counterparts of nearly K\"{a}hler manifolds, see~\cite{G-70},  are used in classifying almost contact metric manifolds.
Any 5-dimensional nearly cosymplectic manifold has an Einstein metric of positive scalar curvature.
For example, the sphere $\mathbb{S}^5$ is endowed with a nearly cosymplectic structure induced by the almost Hermitian structure of~$\mathbb{S}^6$.
In~dimensions greater than 5, nearly cosymplectic manifolds are locally Riemannian products ${\mathbb R}\times\bar M^{2n}$ or $B^5\times\bar M^{2n-4}$, where $\bar M$ is a nearly K\"{a}hler manifold and $B$ is a nearly cosymplectic manifold, see~\cite{NDY-2018}.

\smallskip

The~$f$-structure introduced by K.\,Yano, see \cite{YK-1985}, on a smooth manifold $M^{2n+s}$ serves as a higher-dimensio\-nal analog of {almost complex} ($s=0$) and {almost contact} ($s=1$) structures. This structure is defined by a (1,1)-tensor $f$ of rank $2n$ such that $f^3 + f = 0$.
The tangent bundle $TM$ splits into two complementary subbundles,
$2n$-dimensional $f(TM)$ and $s$-dimensional~$\ker f$.
To~generalize concepts and results from almost contact geometry, 
geometers have studied various broad classes of metric $f$-manifolds (i.e., the distribution $\ker f$ is parallelizable), e.g.,~\cite{CFF-1990,Di-T-2006,CdT-2007}.
 For example, a~metric $f$-manifold is termed an almost ${\cal K}$-manifold if $d \Phi=0$, where $\Phi(X,Y):=g(X,{f} Y)$.
The~distribution $\ker f$ of a ${\cal K}$-manifold is tangent to a $\mathfrak{g}$-foliation \cite{AM-1995} with flat totally geodesic leaves.
An important class of metric $f$-manifolds is given by 
${\cal C}$-mani\-folds, hence $\nabla{f}=0$, see~\cite{b1970}.
Such manifolds are locally products of $\mathbb{R}^s$ and K\"{a}hler manifolds. Nearly ${\cal C}$-manifolds  are defined simi\-larly to nearly K\"{a}hler manifolds, starting from ${\cal C}$-manifolds, i.e., $(\nabla_X f)X=0$, see~\cite{BA-2019}.

\smallskip

\noindent
The question arises: \textit{Under what conditions are nearly $\cal C$-manifolds locally Riemannian~products}?

\smallskip

Recent interest in the $f$-structure among mathematicians is also motivated by the study of dyna\-mics and integration on contact foliations, $\mathfrak{g}$-foliations, and $s$-cosymplectic manifolds, see \cite{AM-1995,Alm-2024,Fin-2024,LSZ-2025}.
Contact foliations generalize to higher dimensions the flow of
the Reeb vector field on contact manifolds, and ${\cal K}$-structures are a particular case of uniform $s$-contact structures.

\smallskip

In \cite{RWo-2}, metric structures on a smooth manifold were introduced that generalize the almost Hermitian structure, the almost contact metric structure, and the metric $f$-structure.
%cosymplectic, Sasakian, etc. metric structures.
These so-called ``weak metric structures" (the~linear complex structure on the contact distribution is replaced by a nonsingular skew-symmetric tensor) 
allow us to revisit classical theory and discover new applications in the dynamics and integration on contact foliations, differential geometry of Killing vector fields, totally geodesic foliations and Ricci-type solitons, 
see \cite{rov2024}, and mathematical physics, see \cite{rst-144}.
In  \cite{rov-128}, weak nearly cosymplectic manifolds were defined,
%(generalizing nearly cosymplectic manifolds), 
and their decomposition was proven under certain conditions, generalizing the classical result of \cite{NDY-2018}.
In ~\cite{rst-61} a special class of weak metric $f$-manifolds, called weak nearly $\cal C$-manifolds, which generalize nearly $\cal C$-manifolds, was studied.

\smallskip

Guided by \cite{rov-128,rst-61}, in the paper we investigate the above question for weak nearly~$\cal C$-manifolds.
We~find conditions \eqref{E-nS-10} and \eqref{E-nS-04c} that are trivially satisfied by almost metric $f$-manifolds and
under which weak nearly ${\cal C}$-manifolds are locally Riemannian products.
In~Section~\ref{sec:01-ns}, following the introductory Section~\ref{sec:00-ns}, we recall necessary results on weak metric $f$-structures. 
%%%%
In~Section~\ref{sec:04-ns}, we generalize some results of~\cite{NDY-2018,rov-128}.
In~Theorem~\ref{T-4.4}, we characterize $(4+s)$-dimensional weak nearly ${\cal C}$-manifolds.
In Theorem~\ref{Th-4.5}, we prove 
%the splitting theorem
that a $(2n+s)$-dimensional weak nearly ${\cal C}$-manifold is locally the Riemannian product of either $\mathbb R^s$ and a weak nearly K\"{a}hler manifold, or, under certain conditions, a weak nearly K\"{a}hler manifold 
$(\bar M^{\,2n-4}, \bar{f},\bar g)$ with the property $\bar\nabla(\bar{f}^{\,2})=0$ and a weak nearly ${\cal C}$-manifold of dimension~$4+s$.
%%%%
Section~\ref{sec:02-ns} contains auxiliary lemmas on the differential geometry 
%and topology 
of weak nearly ${\cal C}$-manifolds.
In particular, Lemma~\ref{L-R01} shows that their contact distribution is curvature invariant.
This yields the following.

\begin{Corollary}%\label{nearly c}
For a nearly ${\cal C}$-manifold 
%satisfying conditions \eqref{E-30b-xi}, \eqref{E-30-xi},
the contact distribution is curvature invariant:
\begin{equation}\label{E-nS-04cc}
 R_{X,Y}Z\in{\cal D}\quad (X,Y,Z\in{\cal D}).
\end{equation}
\end{Corollary}
%\begin{proof}
%For a nearly ${\cal C}$-manifold, $Q = {\rm Id}$ is valid, which guarantees that \eqref{E-nS-10} and \eqref{E-nS-04c} are satisfied. Hence, by \eqref{E-nS-04ccc} of Lemma~\ref{L-R01}, we have
%$R_{X,Y}Z \in \mathcal{D}$ for all $X, Y, Z \in \mathcal{D}$.
%$g(R_{\xi_i, fZ} fX, fY) = 0$ for all $1 \le i \le s$ and $X, Y, Z \in \mathfrak{X}_M$. 
%Therefore, $R_{X,Y}Z \in \mathcal{D}$ for all $X, Y, Z \in \mathcal{D}$, as required.
%\end{proof}

The~following consequences of Theorems~\ref{T-4.4} and \ref{Th-4.5} present new results for nearly~${\cal C}$-manifolds.

\begin{Corollary}%\label{C-4.4}
Let a nearly $\cal C$-manifold $M^{\,2n+s}({f}, \vec\xi, \vec\eta, g)$ satisfy \eqref{E-30b-xi} and \eqref{E-30-xi}. 
If  the coframe 
%$\{\eta^i\}$ 
$\vec\eta$ is uniform, then $n=2$ and $M$ admits an $f$-K-contact structure $(\hat f, \vec\xi, \vec\eta,g)$ with $\hat f=-\nabla\xi_i$.
\end{Corollary}

\begin{Corollary}%\label{Cor-4.5}
Let a nearly ${\cal C}$- (non-${\cal C}$-) manifold 
$(M^{\,2n+s}, {f}, \vec\xi, \vec\eta, g)$
of $n>2$ satisfy \eqref{E-30b-xi} and \eqref{E-30-xi}. 
If~the coframe 
%$\{\eta^i\}$ 
$\vec\eta$ is uniform, then $M$ is locally isometric (the isometry is global if $M$ is complete simply connected) to one of the Riemannian products 
\[
 \mathbb{R}^s\times \bar M^{\,2n},\qquad
 %and 
 B^{\,4+s} \times\bar M^{\,2n-4},
\] 
%where 
the induced structure on $\bar M$ is nearly K\"{a}hler,
and the induced structure on $B$ is a nearly ${\cal C}$- 
structure.
\end{Corollary}

\section{Basic concepts of weak metric $f$-manifolds and examples}
\label{sec:01-ns}

In this section, we will review some well-known concepts and results, see \cite{RWo-2,rov2024,rov-128,rst-43,rst-61}.

Recall that nearly K\"{a}hler manifolds $(M,J,g)$ are defined by A.~Gray \cite{G-70} using the condition that only the symmetric part of $\nabla J$ vanishes, in contrast to the K\"{a}hler case where $\nabla J=0$.

\begin{Definition}%\label{D-wK2}
\rm
A smooth manifold $M^{2n}$ of even dimension equipped with a Riemannian metric $g$ and a skew-symmetric (1,1)-tensor ${f}$ of rank $2n$ 
%such that the tensor ${f}^{\,2}$ is negative-definite, 
is called a \textit{weak almost Hermitian manifold}.
Such $(M, {f}, g)$ is a \textit{weak K\"{a}hler manifold} if $\nabla{f}=0$, where $\nabla$ is the Levi-Civita connection.
%of~$g$.
A~weak almost Hermitian manifold is called a \textit{weak nearly K\"{a}hler manifold}, if
$(\nabla_X{f})X=0$ for all $X\in\mathfrak{X}_M$. 
\end{Definition}

\begin{Definition}
%\label{D-basic}
\rm
A~\textit{weak metric $f$-structure} on a smooth manifold $M^{2n+s}$ $(n,s>0)$ is a set $({f},Q,\vec\xi,\vec\eta$, $g)$, where $f$ is a skew-symmetric $(1,1)$-tensor of rank $2n$, $Q$ is a self-adjoint nonsingular $(1,1)$-tensor,
$\vec\xi=(\xi_1,\ldots,\xi_s)$ are orthonormal vector fields, 
$\vec\eta=(\eta^1,\ldots,\eta^s)$ are dual 1-forms,
and $g$ is a Riemannian metric on $M$, satisfying
\begin{align}\label{2.1}
 &{f}^2 = -Q + \sum\nolimits_{i}{\eta^i}\otimes {\xi_i},\quad {\eta^i}({\xi_j})=\delta^i_j,\quad
%\label{2.1Q-nu}
 Q\,{\xi_i} = {\xi_i}, \\
\label{2.2}
 &g({f} X,{f} Y)= g(X,Q\,Y) -\sum\nolimits_{i}{\eta^i}(X)\,{\eta^i}(Y)\quad (X,Y\in\mathfrak{X}_M).
\end{align}
In this case, $(M^{2n+s}, {f},Q,\vec\xi,\vec\eta,g)$ is called a \textit{weak metric $f$-manifold}.
\end{Definition}

Putting $Y=\xi_j$ in \eqref{2.2}, and using ${\eta^i}({\xi_j})=\delta^i_j$, we get
\begin{align}\label{2.2-eta}
 \eta^j(X) = g(X,\xi_j);
\end{align}
thus, ${\xi_j}$ is orthogonal to the \textit{contact distribution} ${\cal D}:=\bigcap_{\,i=1}^s \ker{\eta^i}$.
For a more intuitive under\-standing of the role of $Q$ in the metric $f$-structure, we explain the following properties:
\[
 {f}\,{\xi_i}=0,\quad {\eta^i}\circ{f}=0,\quad \eta^i\circ Q=\eta^i,\quad [Q,\,{f}]=0 .
\]
By \eqref{2.1}, $f^2\xi_i=0$ is true. From this and \eqref{2.1}, we get
%\begin{align}\label{2.1Q}
 $f^3 + fQ = 0$.
%\end{align}
By this, $Q\xi_i=\xi_i$ and $f^2\xi_i=0$ we get $0=-f^3\xi_i=fQ\xi_i=f\xi_i$.
By $f\xi_i=0$, \eqref{2.2-eta}, and the skew-symmetry of $f$, we get $\eta^i(fX)=g(fX,\xi_i)=-g(X,f\xi_i)=0$.
From this and condition ${\rm rank}\,f=2n$, we conclude that $f$ the distribution ${\cal D}$
of a weak metric $f$-structure is ${f}$-invariant, ${\cal D}=f(TM)$ and $\dim{\cal D}=2n$.
By this and $f^3 + fQ = 0$, we get $f^3X=f^2(fX)=-QfX$; hence, $f^3 + Qf = 0$.
This and $f^3 + fQ = 0$ yield $fQ=Qf$. By~symmetry of $Q$ and $Q\xi_i=\xi_i$, we get $\eta^i(QX)=g(QX,\xi_i)=g(X,Q\xi_i)=g(X,\xi_i)=\eta^i(X)$.
%\smallskip
Therefore, the tangent bundle splits,
$TM = {\cal D}\oplus\ker f$
as the complementary orthogonal sum of its subbundles,
the contact distribution ${\cal D}$, and
$\ker f=\operatorname{span}\{\xi_1,\dots,\xi_s\}$ called the \textit{Reeb distribution}.

The {fundamental $2$-form} $\Phi$ on 
$(M^{2n+s}, {f},Q,\vec\xi,\vec\eta,g)$ is given by
\[
 \Phi(X,Y)=g(X,{f} Y)\quad
 %for all
 (X,Y\in\mathfrak{X}_M).
\]
Recall the co-boundary formulas for exterior derivative $d$ on a 1-form ${\omega}$ and a $2$-form $\Phi$,
\begin{align*}
%\label{E-3.3}
 d\Phi(X,Y,Z) &=  X\,\Phi(Y,Z) + Y\,\Phi(Z,X) + Z\,\Phi(X,Y) \notag\\
 &\ -\Phi([X,Y],Z) - \Phi([Z,X],Y) - \Phi([Y,Z],X),\\
 d\omega(X,Y) &= X({\omega}(Y)) - Y({\omega}(X)) - {\omega}([X,Y])\quad (X,Y\in\mathfrak{X}_M).
\end{align*}
Note that the above equality for $d\Phi$ yields
\begin{align}\label{E-3.3}
 %3\,
 d\Phi(X,Y,Z) = (\nabla_X\,\Phi)(Y,Z)+(\nabla_Y\,\Phi)(Z,X)+(\nabla_Z\,\Phi)(X,Y).
\end{align}

\begin{Definition}\rm
A weak metric $f$-structure is called a \textit{weak $f$-{\rm K}-contact structure} if 
%\begin{align*}%\label{2.3}
 $\Phi=d{\eta^1}=\ldots =d{\eta^s}$ 
%\end{align*}
is valid and all vector fields ${\xi_i}$ are Killing,~i.e.,
%\nonumber
 $(\pounds_{{\xi_i}}\,g)(X,Y)
 %:= {\xi_i}(g(X,Y)) -g([{\xi_i},X],Y) - g(X,[{\xi_i},Y]) \\
 =g(\nabla_X {\xi_i}, Y) +g(\nabla_Y {\xi_i}, X)=0$.
%\end{eqnarray*}
Another important case of a weak metric $f$-structure is 
a weak almost ${\cal K}$-structure, that is $d\Phi=0$.
Its special case, is a {weak almost ${\cal C}$-structure}, that is $\Phi$ and $\eta^i\ (1\le i\le s)$ are closed~forms.
\end{Definition}

For~$s=1$, weak almost ${\cal C}$-manifolds become weak almost cosymplectic manifolds.

For weak almost ${\cal K}$-manifolds (and their subclass, 
weak almost ${\cal C}$-manifolds), the Reeb distribution $\ker f$ is tangent to a foliation.
Moreover,  weak almost ${\cal C}$-manifolds satisfy the following conditions:
\begin{align}
\label{E-30b-xi}
 [\xi_i, \xi_j] & =0\quad (1\le i,j \le s), \\
\label{E-30-xi}
 g(\nabla_{X}\,\xi_i,\ \xi_j) & = 0\quad 
 (X\in\mathfrak{X}_M,\ 1\le i,j \le s),
\end{align}
%for $1\le i,j \le s$.
trivial for~$s=1$.
The~following condition is a consequence of \eqref{E-30-xi}:
\begin{align}\label{E-30c-xi}
 \eta^k(\nabla_{\xi_i}\,\xi_j)=0\quad(1\le i,j,k \le s) .
\end{align}

\begin{Remark}\rm
By~\eqref{E-30b-xi}, the 
%Reeb 
distribution $\ker f$ of a weak almost ${\cal C}$-manifold is tangent to a $\mathfrak{g}$-foliation with an abelian Lie~algebra. 
%\begin{Remark}[\cite{AM-1995}]\rm
A foliation of dimen\-sion $s$ on a smooth connected manifold $M$ is called a 
$\mathfrak{g}$-\textit{foliation} \cite{AM-1995},
where $\mathfrak{g}$ is a Lie algebra of dimension $s$, if there exist complete vector fields $\xi_1,\ldots,\xi_s$ on $M$ which, when restricted to each leaf, form a parallelism of this submanifold with a Lie algebra isomorphic to~$\mathfrak{g}$.
\end{Remark}

A metric $f$-structure, satisfying \eqref{E-30b-xi} and $\nabla f=0$, is a normal weak almost ${\cal C}$-structure, see \cite[Theorem~5]{rst-43}. 
We~consider the case when only the symmetric part of $\nabla f$ vanishes. 

%rov2024,
\begin{Definition}
%\cite{rst-61}]
%\label{Def-2.3}
\rm
 A weak metric $f$-structure is called a \textit{weak nearly ${\cal C}$-structure} if
\begin{equation}\label{E-nS-01b}
 (\nabla_X{f})Y + (\nabla_Y{f})X = 0.
\end{equation}
\end{Definition}

Taking derivative of $g({f} V,Z)=-g(V,{f} Z)$, we see that $\nabla_{Y}{f}$ of a weak nearly ${\cal C}$-manifold is skew-symmetric:
%\begin{equation}\label{E-nS-05e}
 $g((\nabla_{Y}{f}) V, Z)=-g((\nabla_{Y}{f}) Z, V)$.
%\end{equation}
Taking the derivative of this, we see that $\nabla^2_{X,Y}{f}$ is also skew-symmetric:
%\[
 $g((\nabla^2_{X,Y}{f})V, Z)=-g((\nabla^2_{X,Y}{f})Z, V)$,
%\]
where 
$\nabla^2_{X,Y} := \nabla_{X}\nabla_{Y} - \nabla_{\nabla_{X}Y}$.

\begin{Example}%\label{Ex-C-S}
\rm
(i) Any weak K\"{a}hler manifold is obviously weak nearly K\"{a}hler. 
%Several authors studied the problem of finding skew-symmetric parallel 2-tensors (different from almost complex structures) on a Riemannian manifold and classified them, e.g., \cite{H-2022}.
Take two or more (nearly) K\"{a}hler manifolds $(M_j,g_j, J_j)$, where $J_j^2=-{\rm Id}_{\,j}$.
The~pro\-duct $\prod_{j}(M_j,g_j,\sqrt{\lambda_j}J_j)$, where $\lambda_j\ne1$ are different positive constants, is a weak 
(nearly) K\"{a}hler manifold with 
$Q=\bigoplus_{j}\lambda_j{\rm Id}_{j}$.
The~classification of weak nearly K\"{a}hler manifolds in dimensions $\ge 4$ is an open problem.

(ii) To construct a weak nearly $\cal C$-structure
$({f}, Q,\vec\xi,\vec\eta,g)$
on
%For
the Riemannian product $M=\bar M\times\mathbb{R}^s$
of a weak nearly K\"{a}hler manifold $(\bar M, \bar{f}, \bar g)$
with $\Omega(X,Y)=\bar g(X, \bar{f}Y)$ and a Euclidean space $(\mathbb{R}^s, dy^2)$, we take any point $(x, y)$ of $M$ and set
\begin{align*}%\label{E-f-prod}
 \xi_i = (0, \partial_{\,y^i}),\ \
 \eta^i =(0, dy^i),\ \
 {f}(X, \partial_{\,y^i}) = (\bar{f} X, 0),\ \
 Q(X, \partial_{\,y^i}) = (-\bar{f}^{\,2} X, \partial_{\,y^i}),
 % = -{f}^2 + \eta\otimes\xi,
\end{align*}
where $X\in T_x\bar M$. Note that $\nabla f=0$ if and only if $\overline\nabla\bar f=0$.
On the other hand, $\overline\nabla\bar f=0$ if and only if $d\Omega=0$,
see \eqref{E-3.3} with $\Phi=\Omega$,
i.e., $(\bar M,\Omega)$ is a symplectic manifold.

(iii) 
%We present an interesting way to construct a weak nearly ${\cal C}$-structure on a manifold $M$ that has two nearly ${\cal C}$-structures with the same Reeb vector fields. 
Let a Riemannian manifold \((M^{2n+s}, g)\) admit two nearly ${\mathcal C}$-structures with common 
%Reeb 
vector fields \(\xi_i\) and one-forms \(\eta^i = g(\xi_i, \cdot)\) for $1 \le i \le s$. Let \(f_1 \ne f_2\) be such that the self-adjoint tensor
$\psi := f_1 f_2 + f_2 f_1$ is not identically zero.
Then the tensor $f := (\cos t)\,f_1 + (\sin t)\,f_2$
for small constant \(t > 0\) 
%($t$ is a constant) 
satisfies the condition~\eqref{E-nS-01b}. %Indeed,
%\begin{align*}
%    (\nabla_X f)X &= \nabla_X (fX) -f(\nabla_X X)
%%  \\ &  =\nabla_X[(\cos t)\,f_1X + (\sin t)\,f_2X] - ((\cos t)\,f_1 + (\sin t)\,f_2)(\nabla_X X)
%%  \\ & = \cos t\, \nabla_X (f_1X) + \sin t\,\nabla_Xf_2X - \cos t\, f_1 (\nabla_X X) -  \sin t\,f_2(\nabla_X X) \\
% = \cos t\, (\nabla_X f_1)X + \sin t\,(\nabla_X f_2)X = 0.
%\end{align*}
Observe that $f^2=-Q
%{\rm Id} + (\sin t \cos t)\,\psi 
+ \sum\nolimits_{i} \eta^i \otimes \xi_i$ with
$Q= {\rm Id} - (\sin t \cos t)\,\psi$, which is positive definite for small $t>0$. 
Thus, \((f, Q, \vec\xi, \vec\eta, g)\) is a weak nearly $\cal C$-structure on \(M^{2n+s}\).
\end{Example}

Let $^\top$ denote the ${\cal D}$ component of a vector.
The following condition is trivial when $Q={\rm Id}$:
\begin{align} \label{E-nS-10}
 ( (\nabla_X Q)Y )^\top=0, \quad (\nabla_Y Q)Y =0\quad 
 (X \in TM,\ Y \in{\cal D}). %\label{E-nS-100}
\end{align}
%\begin{Remark}\rm \label{rem1}
Using the first equality of \eqref{E-nS-10}, for any $X,Y \in TM$, we have
\begin{align}\label{E-nS-10b}
\notag
 (\nabla_X Q)Y &= [(\nabla_X Q)Y]^\top + \sum\nolimits_i\eta^i((\nabla_X Q)Y)\, \xi_i \notag \\
 %&  = [(\nabla_X Q)Y^\top + \sum\nolimits_i\eta^i(Y) (\nabla_X Q)\,\xi_i]^\top +\sum\nolimits_i g((\nabla_X Q)Y,\xi_i)\notag  \xi_i\\
 & = -\sum\nolimits_i\{\eta^i(Y)\widetilde Q\nabla_X \xi_i 
 + g(\widetilde Q\nabla_X \xi_i,Y) \,\xi_i\}.
\end{align}

\begin{Remark}\rm 
In \cite{rov2024}, instead of the overly strict condition (4), the 
%above 
condition \eqref{E-nS-10} should be used.
\end{Remark}

\begin{Lemma}
%[see Proposition 2 in \cite{rst-61}]
\label{L-nS-01}
 On a weak nearly ${\cal C}$-manifold satisfying \eqref{E-30b-xi} and \eqref{E-30c-xi},
the Reeb distribution $\ker f$ defines a flat totally geodesic foliation;
moreover, if 
%the conditions 
\eqref{E-30-xi} and \eqref{E-nS-10} hold, then the vector fields $\xi_i$ are~Killing.
\end{Lemma}

\begin{proof}
The first part easily follows from  \cite[Proposition 2]{rst-61}. For the second part, following the proof of \cite[Proposition 2]{rst-61}, we have 
$(\mathcal{L}_{\xi_j}\,g)(\xi_k,\cdot)=0$ and 
\[
(\mathcal{L}_{\xi_j}\,g)(X,Y)=\eta^j\big((\nabla_X Q)Y
+(\nabla_Y Q)X\big)\qquad (X,Y\in{\cal D}).
\]
By this and the second equality of \eqref{E-nS-10}, 
we obtain $(\mathcal{L}_{\xi_j}\,g)(X,Y)=0$ for $X,Y\in{\cal D}$.
Combining this with $(\mathcal{L}_{\xi_j}\,g)(\xi_k,\cdot)=0$, we conclude that $\xi_j$ is Killing.
\end{proof}

For a Riemannian manifold $(M,g)$ equipped with a~Killing vector field ${\zeta}$, we get, see \cite{YK-1985},
\begin{equation}\label{E-nS-04}
 \nabla_X\nabla_Y\,{\zeta} - \nabla_{\nabla_X Y}\,{\zeta} = R_{\,X,\,{\zeta}}\,Y ,
\end{equation}
where $R_{{X},{Y}}=[\nabla_X, \nabla_Y] -\nabla_{[X,Y]}$ is the curvature tensor. We have $g(R_{\,\xi, Z}\,{f} X, {f} Y) = 0$
for a nearly cosymplectic manifold, see \cite{E-2005}; thus its 
%contact 
distribution ${\cal D}=\ker\eta$ is {curvature invariant},
see~\eqref{E-nS-04cc}.
%For~example, the property \eqref{E-nS-04cc} is true for any 
%distribution ${\cal D}$ on a real space form.
%
%A nearly ${\cal C}$-manifold satisfying conditions \eqref{E-30b-xi} and \eqref{E-30-xi}, also satisfies \eqref{E-nS-04cc}
%with ${\cal D}=\bigcap_{\,i=1}^s \ker{\eta^i}$, 
%see Corollary~\ref{nearly c}. 
In our study of weak nearly ${\cal C}$-manifolds, we assume a condition weaker than \eqref{E-nS-04cc}: 
\begin{equation}\label{E-nS-04c}
 R_{\widetilde Q X,Y}Z\in{\cal D}\quad (X,Y,Z\in{\cal D}).
\end{equation}
Obviously, conditions \eqref{E-nS-10} and~\eqref{E-nS-04c}  become trivial for metric $f$-manifolds, i.e., when $Q={\rm Id}$.

\section{Main results}
\label{sec:04-ns}

Here, we prove a splitting result (Theorem~\ref{T-4.4}) for weak nearly ${\cal C}$-manifolds satisfying some additional conditions such as \eqref{E-nS-10} and~\eqref{E-nS-04c}.
%that become trivial for metric $f$-manifolds.
Theorem~\ref{Th-4.5} characterizes $(4+s)$-dimensional weak nearly~${\cal C}$-manifolds.

We define the tensor field ${\bf h}=(h_1,\ldots,h_s)$, 
as in the classical case, e.g., \cite{E-2005}, where 
\begin{equation}\label{E-c-01}
  h_i = \nabla\xi_i \quad (1\le i \le s).
\end{equation}
As a consequence of Lemma~\ref{L-nS-01}, the tensor fields $h_i$ are skew-symmetric: 
\[
 g(h_i X,\, X) = g(\nabla_X\,\xi_i, X) = \tfrac12\,(\pounds_{\xi_i}\,g)(X,X) = 0\quad(1 \le i \le s).
\] 
Using \eqref{E-30-xi}, we get 
%\[
 $\eta^j\circ h_i = 0$
 and
 $h_i\,\xi_j = 0$
 for $1 \le i,j\le s$. This shows that $\operatorname{dim (ker} h_i) \ge s$  for all $i$.
%\] 

%\begin{Lemma}\label{L-R003}
% For a weak nearly ${\cal C}$-manifold satisfying \eqref{E-30b-xi}, \eqref{E-30-xi}, \eqref{E-nS-10} and \eqref{E-nS-04c}, the operators $h_i$ and $h_j$ commute for $1 \le i,j \le s$ and each composition $h_ih_j$ is self-adjoint.
%\end{Lemma}

The following result generalizes \cite[Proposition~4.1]{rov-128}
(and \cite[Proposition~4.2]{NDY-2018}) to the case $s>1$.

\begin{Proposition}\label{Prop-4.1}
For a weak nearly ${\cal C}$-manifold with conditions \eqref{E-30b-xi}, \eqref{E-30-xi}, \eqref{E-nS-10} and \eqref{E-nS-04c}, 
the operators $h_i$ and $h_j$ commute and each composition $h_ih_j$ is self-adjoint for $1 \le i,j \le s$.
%%%%%%
Moreover, the eigenvalues and their multiplicities of $h_i h_j$ and $h_i^2$ are constant, and  $h_i h_j$ and $h_i^2$ share the same eigen-frame.
\end{Proposition}

\begin{proof}
From the first Bianchi identity, using \eqref{E-nS-05ccc}, we have $R_{\xi_i, X}\,\xi_j=R_{\xi_j, X}\, \xi_i$ for all $X \in \mathfrak{X}_M$ and $1 \le i,j \le s$. From this, applying \eqref{E-3.24} and \eqref{E-3.23}, we get 
the commutativity of $h_i$ and $h_j$:
\[
 0=  R_{\xi_i, X}\,\xi_j-R_{\xi_j, X}\, \xi_i =-(\nabla_X h_i)\xi_j+ (\nabla_X h_j)\xi_i = h_ih_jX-h_jh_iX.
\]
The first part follows 
%directly as a consequence of 
from the commutativity and skew-symmetry of the operators $h_i\ (1 \le i \le s)$.
%%%%%%%%%%%%%%%%%%%%%%%%%%%%%%

From the above 
%\eqref{E-3.24} 
and Lemma~\ref{L-nS-04} we obtain
\begin{equation}\label{E-nS-11}
 (\nabla_X\, h_i h_j)Y
 = h_i(\nabla_X\, h_j)Y + (\nabla_X\, h_i) h_j Y
 = \sum\nolimits_{k}\big[g(X, h_k h_i h_j Y)\,\xi_k 
 - \eta^j(Y)\,h_i h_j h_k X\big].
\end{equation}
%}
Consider an eigenvalue $\mu$ of $h_i h_j$ and a local unit vector field $Y\in\mathcal{D}$ such that $h_i h_j Y = \mu Y$.
Applying \eqref{E-nS-11} for any nonzero vector fields 
$X,Y\in \mathcal{D}$, we find $g((\nabla_X\, h_i h_j)Y, Y)=0$, thus
\begin{eqnarray*}
 && 0 = g((\nabla_X\, h_i h_j)Y, Y) 
 = g(\nabla_X\,(h_i h_j Y), Y)-g(h_i h_j(\nabla_X\,Y),Y) \\
 &&\ \ = X(\mu)\,g(Y, Y) + \mu\,g(\nabla_X\,Y, Y) 
 - g(\nabla_X\,Y, h_i h_j Y) = X(\mu)\,g(Y, Y),
\end{eqnarray*}
which implies that $X(\mu) = 0$ for all $X\in\mathfrak{X}_M$.
When $j=i$, we get a similar result for $h_i^2$.

Next, by the above,
%Proposition~\ref{Prop-4.1}, 
the operators $h_i$ and $h_j$ commute for all $i,j$, which implies that $h_i^2$ and $h_j^2$ also commute. Each $h_i^2$ is self-adjoint, and therefore, diagonalizable.

It is well known that 
%two 
a finite family of
diagonalizable operators on a finite-dimensional Euclidean space is simultaneously diagonalizable if and only if the opeators commute. 
%This fact extends by mathematical induction to any finite family of commuting diagonalizable operators. 
%This can be proved by mathematical induction for any finite family of commuting diagonalizable operators.
By the above, there exists an orthonormal basis $\{e_{k}\}$ of %the tangent space 
$TM$ such that $h_i^2 e_{k} = \lambda_{i,k}\, e_{k}$ for all $1 \le i \le s$ and~$k$.
\end{proof}

\smallskip

By Proposition~\ref{Prop-4.1}, the spectrum of the self-adjoint operator $h_i^2$ has the~form
\begin{equation}\label{E-nS-11b}
 Spec(h_i^2) = \{0, -\lambda_{i,1}^2,\ldots 
 -\lambda_{i,r_i}^2\} \qquad (1 \le i \le s),
\end{equation}
%and similarly, for $Spec(h_ih_j)$,
where $r_i\ge1$, $\lambda_{i,j}$ is a positive real number and $\lambda_{i,j}\ne \lambda_{i,k}$ for $j\ne k$.
If $X\ne0$ is an eigenvector of $h_i^2$ with eigenvalue $-\lambda^2_{i,j}$, then $X, {f} X, h_iX$ and $h_i\,{f} X$ are orthogonal nonzero eigenvectors of $h_i^2$ with the same eigenvalue $-\lambda^2_{i,j}$.
Since $h_i(\xi_j)=0$ for all $1 \leq i,j \le s$, the Reeb  distribution $\ker f=\operatorname{span}\{\xi_1,\dots,\xi_s\}$ is contained in the kernel of $h_i$. Consequently, the eigenvalue $0$ of $h_i^2$ has multiplicity $2p_i+s$ for some integer $p_i\geq 0$, associated with each operator $h_i^2$. 
Denote by 

$\bullet$ $D_{i,0}$ the smooth distribution of the eigenvectors of $h_i^2$ with eigenvalue $0$ orthogonal to $\ker f$;

$\bullet$ $D_{i,j}\ (1\le j\le r_i)$ be the smooth distribution of the eigenvectors of $h_i^2$ with eigenvalue $-\lambda^2_{i,j}$.

\noindent
Note that 
%the distributions 
$D_{i,0}$ and $D_{i,j}$ belong to ${\cal D}$ and are ${f}$-
%invariant 
and $h_k$-invariant for each $1 \le i,k \le s$. 
\begin{Remark}\rm
 By the skew-symmetry of $h_i$, for every $X,Y \in \mathfrak{X}_M$, we have
\begin{align}\label{E-c-01b}
%\notag
 d\eta^i(X,Y) 
%&=X(\eta^i(Y))-Y(\eta^i(X))-\eta([X,Y]) \\
%\notag    
% & = g(\nabla_X Y,\xi_i)+g(Y, \nabla_X \xi_i) - g(\nabla_Y X,\xi_i)-g(X, \nabla_Y \xi_i) - g([X,Y],\xi_i) \\
% & = g(Y,h_iX)-g(X,h_iY) 
   = 2\,g(h_iX,Y) \quad (1 \le i \le s).
\end{align}
The above shows that $\operatorname{ker} h_i = \operatorname{ker}d\eta^i$ and as a consequence of this, $\operatorname{dim\, (ker} d\eta^i) \ge s $ for all $i$.
\end{Remark}

The following result generalizes \cite[Proposition~4.3]{rov-128}
(and \cite[Proposition~4.1]{NDY-2018}) to the case~$s>1$.
%and does not use Lemmas~\ref{L-R02}--\ref{L-R03}.

\begin{Proposition}\label{Th-4.1}
For a weak nearly ${\cal C}$- (non-weak--${\cal C}$-) manifold $(M^{\,2n+s}, {f},Q,\vec\xi,\vec\eta,g)$
with conditions \eqref{E-30b-xi} and  \eqref{E-30-xi}, 
the equality ${\bf h}\equiv0$ is valid if and only if the manifold is locally 
%isometric to 
the Riemannian product of 
$\mathbb{R}^s$
%an $s$-dimensional Euclidean space 
and a weak nearly K\"{a}hler (non-weak-K\"{a}hler) manifold.
\end{Proposition}

\begin{proof}
By \eqref{E-c-01b} and the condition ${\bf h} \equiv 0$, we have $d\eta^i=0$ for all $i$. 
We have 
\[
 0=d\eta^i(X,Y)=-\eta^i([X,Y])\quad(1\le i\le s,\  X,Y\in{\cal D}),
\] 
which indicates that $[X,Y] \in{\cal D}$; thus the 
%contact 
distribution ${\cal D}$ is integrable. Integral manifolds of ${\cal D}$ are totally geodesic as
\[
 g(\nabla_X\,Y, \xi_i)=-g(Y, h_iX)=0\quad
 (1\le i\le s,\ X,Y\in{\cal D}).
\]
Since $\operatorname{ker} f$ is integrable, totally geodesic and flat (see Lemma \ref{L-nS-01}), by de Rham Decomposition Theorem, 
%e.g., \cite{KN-69}, 
the manifold is locally 
%isometric to 
the Riemannian product $\mathbb{R}^s\times\bar M$.
The~weak nearly ${\cal C}$-structure on $M$ induces on $\bar M$ a weak nearly K\"{a}hler structure.

Conversely, if a weak nearly ${\cal C}$-manifold is locally 
%isometric to 
the Riemannian product $\mathbb{R}^s\times\bar M$, where $\bar M$ is a weak nearly K\"{a}hler manifold and $\xi_i = (\partial_{x_{i}},0)$ ($(x,y) \in \mathbb{R}^s\times\bar M$), then $d\eta^i=0$ for $1\le i\le s$ or simply ${\bf h}\equiv0$, by \eqref{E-c-01b}. 
By this and the definition of $\xi_i$, the conditions \eqref{E-30b-xi} and  \eqref{E-30-xi} are true. 
\end{proof}

The following result generalizes \cite[Proposition~4.2]{rov-128} (and \cite[Proposition~4.3]{NDY-2018}) to the case $s>1$.

\begin{Proposition}\label{P-4.3}
Let $(M^{\,2n+s}, {f},Q,\vec\xi,\vec\eta,g)$ be a weak nearly $\mathcal{C}$-manifold with conditions \eqref{E-30b-xi}, \eqref{E-30-xi}, \eqref{E-nS-10} and \eqref{E-nS-04c}.
Then
\begin{itemize}[noitemsep]
 \item[(i)] \vskip-2mm
 distributions $\ker f\oplus{\cal D}_{i,j}\ (1 \le i \le s,\,1 \le j \le r_i )$ are integrable with totally geodesic leaves.
\end{itemize}
\vskip-2mm
Moreover, if 
%$0$ is an eigenvalue of $h_i^2$ with multiplicity $>s$ and 
%the kernel of 
$\dim(\ker d\eta^i)>s$ for all $i$,
%has dimension greater than $s$, 
then the following is true: 
\begin{itemize}[noitemsep]
\item[(ii)] \vskip-2mm
${\cal D}_{i,0}\ne0$ for all $i$, and
distributions $\operatorname{ker}f\oplus{\cal D}_{i,0}$ are integrable with totally geodesic leaves;
%%%%%%%%%%%%%%%
 \item[(iii)] if the exterior derivatives of the coframe %$\{\eta^i\}$
 $\vec\eta$ have a common kernel, then  distributions ${\cal D}_{i,0}$ coincide, are integrable with totally geodesic leaves, and the induced structure on each leaf of ${\cal D}_{i,0}$ is a weak nearly K\"{a}hler structure $(\bar{f}, \bar g)$ with the property $\bar\nabla(\bar{f}^{\,2})=0$;
%%%%%%%%%%%%%%%
\item[(iv)] the distribution 
$\ker f\oplus{\cal D}_{i,1}\oplus\ldots\oplus{\cal D}_{i,r_i}$ is integrable with totally geodesic leaves.
\end{itemize}

\end{Proposition}

\begin{proof}
Consider an eigenvector $X$ of $h_i^2$ with eigenvalue $-\lambda^2_{i,j}$, see \eqref{E-nS-11b}. 
Then $\nabla_X\,\xi_k = h_kX \in{\cal D}_{i,j}$ ($k=1,\ldots,s$), as $h_i^2(h_kX)=h_k(h_i^2X)=-\lambda^2_{i,j}\,h_kX$. 
%On the other hand, 
Since \eqref{E-nS-11} implies $\nabla_{\xi_k} h_i^2=0$, then $\nabla_{\xi_k} X$ is also an eigenvector of $h_i^2$ with eigenvalue $-\lambda^2_{i,j}$ for all $k=1,..,s$. By Lemma \ref{L-nS-01}, the distribution $\operatorname{ker} f$ defines a flat totally geodesic foliation. Moreover, since $\nabla_X\,\xi_k,\ \nabla_{\xi_k} X \in \mathcal{D}_{i,j}$ for all $X \in \mathcal{D}_{i,j}$, we have $\operatorname{ker} f\oplus\mathcal{D}_{i,j} $ is integrable (involutive).
Now, taking $X, Y \in{\cal D}_{i,j}$ and applying \eqref{E-nS-11}, we get
\[
 h_i^2(\nabla_X\,Y) = -\lambda^2_{i,j}\,\nabla_X Y - (\nabla_X\,h_i^2)Y = -\lambda^2_{i,j}\,\nabla_X Y + \lambda^2_{i,j}\sum\nolimits_{k}g(X, h_iY)\,\xi_k.
\]
%\[
Using the above and the anti-commutativity of $f$ and $h_i$, 
we have 
\[
 h_i^2({f}^2\nabla_X Y) = {f}^2(h_i^2\nabla_X Y) 
 = -\lambda^2_{i,j}\,{f}^2(\nabla_X Y);
\]
hence, ${f}^2\nabla_X Y \in{\cal D}_{i,j}$.
Similarly, using \eqref{E-nS-01d}, we get $\widetilde Q\nabla_X Y \in{\cal D}_{i,j}$.
It follows that
\[
 \nabla_X Y = -\widetilde Q\,\nabla_X Y -{f}^2\nabla_X Y +\sum\nolimits_{k} \eta^k(\nabla_X Y)\,\xi_k,
\]
see \eqref{2.1}, belongs to the distribution $\ker f\oplus{\cal D}_{i,j}$. 
%Also since 
This proves $(i)$.

\smallskip

Let $\dim(\ker d\eta^i)>s$ for all $i$, i.e., the multiplicity of the eigenvalue 0 of $h_i^{\,2}$ be greater than~$s$. 
Similarly to $(i)$, we obtain
$\nabla_X\xi_k, \nabla_{\xi_k} X\in\ker f\oplus \mathcal{D}_{i,0}$ for $X \in \ker f \oplus \mathcal{D}_{i,0}$ and $1 \le k \le s$.
Hence, the distribution $\ker f \oplus \mathcal{D}_{i,0}$ is integrable.
Moreover, by~\eqref{E-nS-11} we have 
$(\nabla_X h_i^{\,2})Y = 0,$ for
  $X,Y \in \operatorname{ker}f \oplus \mathcal{D}_{i,0}.$
By the above, $h_i^{\,2}(\nabla_X Y)=0$ for $1 \le i \le s$, so $\ker f \oplus \mathcal{D}_{i,0}$ is totally geodesic, proving~$(ii)$.

Since $ \mathcal{D}_{i,0}\ne0$,
%also contains vectors having eigenvalue 0, 
we get $h_i^{\,2}(\nabla_X Y)=0$ for all $X,Y \in \mathcal{D}_{i,0}$. In~addition, we compute
\[
g(\nabla_X Y, \xi_k)
  = -g\big(Y, \nabla_X \xi_k\big)
  = -g(Y, h_k X),
\]
which vanishes if 
%and only if
$\operatorname{ker} h_i =\operatorname{ker} h_k$ for all $1 \le i,k \le s$. 
By~\eqref{E-c-01b}, this condition is equivalent to the fact that the exterior derivatives of the coframe~$\vec\eta$
%$\{\eta^i\}$ 
have a common kernel.
Thus, each distribution $D_{i,0}$ defines a totally geodesic foliation if and only if the exterior derivatives of the coframe $\vec\eta$
%$\{\eta^i\}$ 
have a common kernel.
By \eqref{E-nS-01b} and 
%\eqref{E-nS-01d} of 
Lemma~\ref{L-nS-02}, the leaves of $D_{i,0}$ are ${f}$-invariant and $Q$-invariant.
Thus, the weak nearly ${\cal C}$-structure on $M$ with conditions \eqref{E-30b-xi}, \eqref{E-30-xi}, \eqref{E-nS-10} and \eqref{E-nS-04c}
induces a weak nearly K\"{a}hler structure $(\bar{f}, \bar g)$ on each leaf of $D_{i,0}$ with the property $\bar\nabla(\bar{f}^{\,2})=0$,
where $\bar\nabla$ is the Levi-Civita connection of $\bar g$. This proves~$(iii)$.

\smallskip

To establish $(iv)$, it suffices to verify that for $X \in{\cal D}_{i,j}$ and $Y \in{\cal D}_{i,k}$ with $j\ne k$, one has $\nabla_X Y\in\ker f\oplus {\cal D}_{i,1}\oplus\cdots\oplus{\cal D}_{i,r}$. 
Similarly to part $(i)$, we obtain 
\[
 h_i^{2}(\nabla_X Y) = -\lambda_{i,k}\,\nabla_X Y 
+\lambda_{i,k}\sum\nolimits_{j} g(X,h_j Y)\,\xi_j;
\]
hence, $h_i^{2}(f^{2}\nabla_X Y) = -\lambda_{i,k}\, f^{2}\nabla_X Y$. 
Thus $f^{2}\nabla_X Y \in{\cal D}_{i,k}$, and we calculate $\widetilde Q\nabla_X Y \in{\cal D}_{i,k}$. Combi\-ning these observations, we may write
\[
\nabla_X Y = -\widetilde Q(\nabla_X Y) - f^{2}\nabla_X Y + \sum\nolimits_{k=1}^{s} \eta^{k}(\nabla_X Y)\,\xi_k \in \operatorname{ker}f \oplus {\cal D}_{i,k}\subset \ker f\oplus {\cal D}_{i,1}\oplus\cdots\oplus{\cal D}_{i,r}.
\]
Similarly, we obtain $\nabla_Y X\in\operatorname{ker}f \oplus{\cal D}_{i,j}$. 
Consequently, both $[X,Y]$ and $\nabla_X Y$ belong to 
$\ker f \oplus{\cal D}_{i,1}\oplus \cdots \oplus {\cal D}_{i,r}$, which completes the proof.
\end{proof}

\begin{Definition}[see \cite{Fin-2024}]\label{D-05}
\rm
Let $M$ be a $(2n+s)$-dimensional smooth manifold where  $n, s$ are positive integers. 
An $s$-\textit{contact structure} on $M$ is a collection 
$\vec{\eta} = (\eta^1, \cdots, \eta^s)$ of $s>1$ pointwise linearly independent non-vanishing one-forms,
together with a splitting $TM = {\mathcal R} \oplus{\cal D}$
of the tangent bundle, satisfying the following conditions:
\begin{itemize}[noitemsep]
 \item[(i)]\vskip-1.5mm ${\cal D} := \bigcap_{\,i}\ker\eta^i$;
 \item[(ii)] $d\eta^i|_{\cal D}$ is non-degenerate
 for every $i$;
 \item[(iii)] $\ker d\eta^i = {\mathcal R}$ for every $i$.
\end{itemize}
\vskip-1.5mm
A manifold $M$ endowed with such a structure is called an $s$-\textit{contact manifold} and is denoted by 
$(M, \vec{\eta}, {\mathcal R}\oplus{\cal D})$. 
A coframe $\vec\eta$
%$\vec{\eta}$ 
(and the $s$-contact structure it defines) is said to be \textit{uniform} if it satisfies 
\begin{itemize}
\item[(iv)] \vskip-2mm
 $d\eta^i=d\eta^j$ for all $1 \le i,j \le s$.
\end{itemize}
\end{Definition}

An example of $s$-contact manifolds is a weak $f$-{\rm K}-contact manifold, in fact, it is uniform.
By~\eqref{E-c-01b}, $d\eta^i=d\eta^j \iff h_i=h_j$ for all $1 \le i,j \le s$.

\smallskip

Next, we will use Lemmas~\ref{L-R02}--\ref{T-new}
%{L-R03} 
to characterize $(4+s)$-dimensional weak nearly ${\cal C}$-manifolds. 

The following result generalizes \cite[Theorem~4.2]{rov-128}
(see also Theorem~4.4 in \cite{NDY-2018}) to the case $s>1$.

\begin{Theorem}\label{T-4.4}
Let $(M^{\,2n+s}, {f}, Q, \vec\xi, \vec\eta, g)$ be a weak nearly $\cal C$-manifold satisfying \eqref{E-30b-xi},  \eqref{E-30-xi}, \eqref{E-nS-10} and \eqref{E-nS-04c}.
%Suppose that $h_1=\ldots=h_s$ and $0$ is an eigenvalue of $h_i^2$ with multiplicity $s$. 
If the coframe $\vec\eta$
%$\{\eta^i\}$ 
%defines a uniform $s$-contact structure
has a common exterior derivative, see (iv) of Definition~\ref{D-05}, then $n=2$ and~$M$ admits a weak $f$-K-contact structure 
%%%
$(\hat f,\hat Q, \vec\xi, \vec\eta,g)$ defined by 
%$\hat\eta^i = g(\cdot, \xi_i)$, 
$\hat f = -\nabla \xi_i$ and $\hat QX=R_{\xi_i}X\ (X \in{\cal D})$.
\end{Theorem}

\begin{proof}
We consider 2-forms $\Phi^{(k)}_{i}(X,Y)=g({f} h_i^k X, Y)$ and $\Phi^{(1)}_{i,j}(X,Y)=g({f} h_ih_j X, Y)$, where $i=1,\dots s,\ k = 0, 1$. Also, since $h_i$ and $h_j$ commute, we have $\Phi^{(1)}_{i,j} = \Phi^{(1)}_{j,i}$.
It is easy to calculate
%\[
  $d\Phi(X,Y,Z) = g((\nabla_X {f})Z,Y) +g((\nabla_Y\,{f})X, Z) +g((\nabla_Z\,{f})Y, X)$,
%\]
see \cite{RP-2}.
We will show that
\begin{equation}\label{E-c-03}
 %d\Phi_0 = 3\,\eta\wedge \Phi_1,\quad d\Phi_1 = 3\,\eta\wedge \Phi_2
 d\Phi^{(0)}_{i} = 3 \sum\nolimits_{j}\eta^j\wedge \Phi^{(1)}_{j},\quad
 d\Phi^{(1)}_{i} = 3\sum\nolimits_{j}\eta^j\wedge \Phi^{(1)}_{i,j}.
\end{equation}
Indeed, applying \eqref{E-c-01} and ${f}\,\xi_i=0$, we find the $\xi_i$-component of $(\nabla_X {f}) Y$:
\begin{equation}\label{E-c-03xi}
  g((\nabla_X {f}) Y, \xi_i) = - g((\nabla_X {f})\,\xi_i, Y) = g({f}\nabla_X\,\xi_i, Y) = g({f} h_i X, Y).
\end{equation}
Replacing $Z$ by ${f} Z$ in \eqref{E-3.50} and using \eqref{E-nS-01b}, we obtain for all $1 \le i \le s$
\begin{align}\label{E-3.50b}
 -g((\nabla_{X} {f})Y, {f}^2Q h_j Z) = \sum\nolimits_{i} \big[&\eta^i(X)g(h_j Y,h_ifZ){-}\eta^i(Y)g(h_jX,h_ifZ)\notag \\ 
 &+\eta^i(X)g(\widetilde Q(Q{+} {\rm Id})Y,h_ih_jfZ)\big].
\end{align}
By conditions, ${\bf h}\not \equiv0$ on ${\cal D}\setminus\{0\}$ and $Q(\mathcal D)=\mathcal{D}$, thus from \eqref{E-3.50b} we get
%\[
 $g((\nabla_{X}\,{f})Y, V) =0$ for $X,Y,V\in{\cal D}$.
%\]
By~the above and \eqref{E-c-03xi}, using $X=X^\top+\sum\nolimits_{i}\eta^i(X)\,\xi_i$ and $Y=Y^\top+\sum\nolimits_{i}\eta^i(Y)\,\xi_i$, we obtain
\begin{align}\label{E-c-03b}
\nonumber
 & g((\nabla_{X}\,{f})Y, V)\notag \\ 
 & =\sum\nolimits_{i}  \big[ \eta^i(V)\,g((\nabla_{X^\top}\,{f})Y^\top, \xi_i) \nonumber
 +\eta^i(X)\,g((\nabla_{\,\xi_i}\,{f})Y^\top, V) +\eta^i(Y)\,g((\nabla_{X^\top}\,{f})\,\xi_i, V)\big] \\
\nonumber
 & = \sum\nolimits_{i}  \big[ -\eta^i(V)\,g((\nabla_{X^\top}\,{f})\,\xi_i, Y^\top)
 -\eta^i(X)\,g((\nabla_{Y^\top}\,{f})\,\xi_i, V) +\eta^i(Y)\,g((\nabla_{X^\top}\,{f})\,\xi_i, V)  \big]\\
 & =  \sum\nolimits_{i}  \big[\eta^i(V)\,g({f} h_i X, Y) +\eta^i(X)\,g({f} h_i Y, V) +\eta^i(Y)\,g({f} h_i V, X) \big].
\end{align}
Using the expression  for $d\Phi_{i}^{(0)}$ and \eqref{E-c-03b}, we have the first equality of \eqref{E-c-03}.
%that $d\Phi^{(0)}_{i} = \sum\nolimits_{i}\eta^i\wedge \Phi^{(1)}_{i}$.
%%%%%%%%%%%%

Similarly, using \eqref{E-c-03b} and \eqref{E-3.23b} of Lemma~\ref{L-nS-04}, we get
\begin{align*}
& g((\nabla_X\,({f}\,h_i))Y, Z) = g((\nabla_X {f})h_iY, Z) + g({f}(\nabla_X\,h_i)Y, Z) \\
 &=  \sum\nolimits_{j}  \big[ \eta^j(X)\,g({f}\,h_ih_j Y, Z) +\eta^j(Y)\,g({f}\,h_ih_j Z, X) +\eta^j(Z)\,g({f}\,h_ih_j X, Y)\big],
\end{align*}
which implies the second equality of \eqref{E-c-03}.
%$d\Phi^{(1)}_{i} = 3\sum\nolimits_{j}\eta^j\wedge \Phi^{(2)}_{i,j}$ and completes the proof of \eqref{E-c-03}.
From \eqref{E-c-03}, using the symmetry: 
$\Phi^{(1)}_{i,j}=\Phi^{(1)}_{j,i}$ and the anti-symmetry
of wedge products: $\eta^j\wedge\eta^i=-\eta^i\wedge\eta^j$,
% One is $d^2 p=0$ for any $k$-form $p$, other is anti-symmetry of wedge products. Because of the anti-symmetry of wedge products and symmetry of $\Phi^{(1)}_{i,j}$, the second term in (78) vanishes.
we obtain
\begin{equation}\label{E-Nic-27}
 0=   d^2\Phi^{(0)}_{i} =
  3\sum\nolimits_{j} \big[d\eta^j\wedge \Phi^{(1)}_{j} - \eta^j\wedge d\Phi^{(1)}_{j} \big]=3\sum\nolimits_{j} d\eta^j\wedge \Phi^{(1)}_{j}.
\end{equation}
%Since the coframe $\{\eta^i\}$ have a common the exterior derivative, $d\eta^i=d\eta^j$ for all $i,j$.
Substituting $Y = \xi_i$ into \eqref{E-3.23b} and applying \eqref{E-nS-01a}, we represent the $\xi$-sectional curvature as
\begin{equation*}
 K(\xi_i, X) = g(h_i X, h_i X) = \|h_i X\|^2 
 \quad (X \in \mathcal{D},\ g(X, X) = 1,\ 1 \le i \le s).
\end{equation*}
Under the assumptions, the $\xi$-sectional curvature of $(M,g)$ is positive, which implies that the Jacobi operators 
$R_{\,{\xi_i}}: X \mapsto R_{X,\,{{\xi_i}}}\,{{\xi_i}}\ (1 \le i \le s)$ are positive definite on $\mathcal{D}$. 
By \cite[Theorem~2]{rov2024}, there exists a weak $f$-K-contact structure $(\hat f,\hat Q, \vec\xi, \vec\eta,g)$ on $M$ defined by 
%$\hat\eta^i(\cdot) = g(\cdot, \xi_i)$, 
$\hat f = -\nabla \xi_i$ and $\hat QX=R_{\xi_i}X\ (X\in{\cal D})$. 

%This indicates that $\eta^i$ and $\tilde \eta^i$ coincide and 
As a consequence of this, $\eta^i$ is a contact form and $d\eta^i$ is non-degenerate on $\mathcal D$. Thus, $M$ is a uniform $s$-contact manifold, see Definition~\ref{D-05}.
Now, if $\dim M > 4 + s$, then (since $M $ is a uniform $s$-contact manifold) \eqref{E-Nic-27} implies that 
$d\eta^j \wedge \sum_{j} \Phi^{(1)}_j = 0\, \text{for all } j.$
By~Lemma~\ref{T-new}, this condition forces $\Phi^{(1)}_i = 0$ for all $i$; hence, ${\bf h} \equiv 0$. This contradicts the fact that $d\eta^i$ is non-degenerate on $\mathcal{D}$. Therefore, $\dim M \le 4 + s$.
Finally, since $d\eta^i$ is non-degenerate on $\mathcal{D}$, the distribution $\operatorname{ker} d\eta^i$ has dimension $s$. 
We conclude that $h_i^2|_{\cal D}$ is non-zero and has at least one non-zero eigenvalue, say, $-\lambda^2$, whose multiplicity is $4$: if $X$ is an eigenvector of $h_i^2$ with eigenvalue $-\lambda^2$, then $\{X, fX,h_iX, h_ifX\}$ are also orthogonal eigenvectors of $h_i^2$ with the same eigenvalue. Therefore, $M$ is a $(4+s)$-dimensional manifold and the multiplicity of the eigenvalue $-\lambda^2$ of $h_i^2$ is $4$.
\end{proof}

\begin{Remark}\rm
%It should be noted that 
The weak $f$-K-contact structure 
in Theorem~\ref{T-4.4} does not coincide with the given weak nearly $\mathcal{C}$-structure $({f}, Q, \vec\xi, \vec\eta, g)$, even though the $1$-forms $\eta^i$ are identical. 
Nevertheless, each $\eta^i$ is a contact form, thus $d\eta^i$ is non-degenerate on $\mathcal D$. 
\end{Remark}

%%% Corollary 2

We generalize the splitting Theorem~4.1 in \cite{rov-128}
(see also Theorem~4.5 in \cite{NDY-2018}) to the case $s>1$. 

\begin{Theorem}\label{Th-4.5}
Let $(M^{\,2n+s}, {f}, Q, \vec\xi, \vec\eta, g)$ be a weak nearly ${\cal C}$- (non-weak--${\cal C}$-) manifold of $n>2$
with conditions \eqref{E-30b-xi}, \eqref{E-30-xi}, \eqref{E-nS-10} and \eqref{E-nS-04c}. 
If the coframe $\vec\eta$ 
has a common exterior derivative, see (iv) of Definition~\ref{D-05},
%is uniform, 
then $M$ is locally isometric (the isometry is global if $M$ is complete and simply connected) to one of the Riemannian products 
\begin{align*}%\label{Eq-MM}
 (i)~\mathbb{R}^s\times \bar M^{\,2n},\quad 
 (ii)~B^{\,4+s} \times\bar M^{\,2n-4},
\end{align*}
where the induced structure on $\bar M$ is a weak nearly K\"{a}hler structure $(\bar{f},\bar g)$ satisfying $\bar\nabla(\bar{f}^{\,2})=0$, and the induced structure on $B$ 
is a weak nearly ${\cal C}$- (non-weak--${\cal C}$-) structure satisfying~\eqref{E-nS-10} and~\eqref{E-nS-04c}.
\end{Theorem}

\begin{proof}
If ${\bf h}\equiv0$, then by Proposition~\ref{Th-4.1}, $M$ is locally isometric to $\mathbb{R}^s\times \bar M^{\,2n}$.

By the assumptions and \eqref{E-c-01b}, we have  $h_i=h_j$ for all $i$ and $j$.
Let $h_i\ne0$ on ${\cal D}\setminus\{0\}$ for all $i$, then  
$Spec(h_i^2) = \{0, -\lambda_{i,1}^2,\ldots 
-\lambda_{i,r_i}^2\} $, see \eqref{E-nS-11b},
where $r_i\ge1$ and each $\lambda_{i,k}$ is a positive number. 
If~$\dim\ker h_i=s$, then by Theorem \ref{T-4.4}, $n=2$ which contradicts the assumption that $n>2$. So, we must have $\dim\ker h_i >s$.
By Proposition~\ref{P-4.3}\,$(iii, iv)$, and according to the de Rham Decomposition Theorem, 
%e.g., \cite{KN-69}, 
$M$ is locally isometric to the Riemannian product $B\times\bar M$. Here, $B$ corresponds to integral manifolds of the distribution 
$\ker f\oplus{\cal D}_{i,1}\oplus\ldots\oplus{\cal D}_{i,r_i}$, 
and $\bar M$ corresponds to integral manifolds of $D_{i,0}$,
which is endowed with a weak nearly K\"{a}hler structure $(\bar{f},\bar g)$ and, by the condition \eqref{E-nS-10},
has the property $\bar\nabla(\bar{f}^{\,2})=0$.
Note that $B$ is endowed with an induced weak nearly ${\cal C}$- (non-weak--${\cal C}$-) structure, for which $0$ is an eigenvalue of $h_i^2$ with multiplicity $s$ and the coframe $\vec\eta$ have a common exterior derivative.
By~Theorem \ref{T-4.4}, $B$ is a $(4+s)$-dimensional manifold and consequently, $\dim\bar M=2n-4$.
If~the manifold $(M,g)$ is complete and simply connected, then we apply the de Rham Decomposition Theorem to get the global isometry.
\end{proof}

\begin{Remark} \rm
In aim to obtain an example illustrating Theorem~\ref{Th-4.5}, we ask the following question: \textit{Does there exist a nontrivial 
%weak nearly ${\cal C}$-manifold with $s>1$, that is the product of a 
weak nearly ${\cal C}$-manifold $B^{\,4+s}$ with $s>1$
%and a weak nearly K\"{a}hler manifold $\bar M^{2n-4}$
}?
\end{Remark}

\section{Auxiliary results}
\label{sec:02-ns}

In this section, we consider weak nearly $\mathcal{C}$-manifolds satisfying some of conditions \eqref{E-30b-xi}, \eqref{E-30-xi}, \eqref{E-nS-10} and~\eqref{E-nS-04c}. 
We prove 
%several 
lemmas generalizing known properties of weak nearly cosymplectic~manifolds.
%for the case of $s>1$.
%We~prove several lemmas that generalize for $s>1$ some well-known properties of weak nearly cosymplectic manifolds.

The following result generalizes \cite[Lemma~3.1]{rov-128}
%(and Lemma~3.1 in \cite{E-2005}) 
for the case of $s>1$.

\begin{Lemma}\label{L-nS-02}
For a weak nearly ${\cal C}$-manifold 
$(M^{\,2n+s}, {f},Q,\vec\xi,\vec\eta,g)$ satisfying \eqref{E-30b-xi} and \eqref{E-30-xi}, we get
\begin{align}\label{E-nS-01a}
 & (\nabla_X\,h_i)\,\xi_j = -h_ih_j X \quad (1 \le i,j\le s),\\
\label{E-nS-01c}
 & (\nabla_X {f})\,\xi_i = -{f}\,h_i X \quad (1 \le i\le s),\\
\label{E-nS-01bb}
 & h_i\,{f} + {f}\,h_i =0\quad (1 \le i\le s)\quad (h_i\ {\rm anticommutes\ with}\ {f}),\\
%%%
\label{E-nS-01d}
 & h_i\,Q = Q\,h_i \quad (1 \le i\le s) \quad (h_i\ {\rm commutes\ with}\ Q).
\end{align}
\end{Lemma}

\begin{proof}
Differentiating the equality $h_i\,\xi_j=0$ and using
%the definition
\eqref{E-c-01}, we obtain \eqref{E-nS-01a}:
\[
 0=\nabla_X\,(h_i\,\xi_j) = (\nabla_X\,h_i)\,\xi_j +h_i(\nabla_X\,\xi_j) =(\nabla_X\,h_i)\,\xi_j +h_ih_j X.
\]
Equation \eqref{E-nS-01c} is obtained by differentiating the relation $f\xi_i=0$ and using the definition of $h_i$.
Differentiating 
%the equality 
$g({f}\,Y,\xi_i)=0$ yields 
%\[
 $0= g((\nabla_X {f})Y, \xi_i) +g({f}\,Y, h_i X)$.
%\]
Interchanging $X$ and $Y$ in the last equation, then adding the resulting expression to the original and applying \eqref{E-nS-01b}, yields \eqref{E-nS-01bb}.
Finally, from \eqref{2.1} and \eqref{E-30-xi} we get $f^2h_i X-h_i f^2X=-Qh_iX+h_iQX$ for all $X \in \mathfrak{X}_M$, and \eqref{E-nS-01d} follows from this by applying~\eqref{E-nS-01bb}.
\end{proof}

The following lemmas generalize certain formulas in 
\cite[Lemmas~3.2-3.5]{rov-128}  
%(and Lemmas~3.2 -- 3.5 in \cite{E-2005}) 
for the case of $s>1$.

\begin{Lemma}\label{L-R02}
For a weak nearly ${\cal C}$-manifold satisfying \eqref{E-nS-10} we obtain
\begin{align}\label{E-3.29}
 g((\nabla_X {f}){f} Y, Z) &= g((\nabla_X{f})Y, {f} Z) +\sum\nolimits_{i}\big[ \eta^i(Y) g(h_iX, QZ) +\eta^i(Z) g(h_i X, Q Y)\big] ,\\
\label{E-3.30}
 g((\nabla_{{f} X}{f})Y, Z)&= g((\nabla_X {f})Y, {f} Z) +\sum\nolimits_{i}\big[ \eta^i(X) g(h_iZ, QY) + \eta^i(Z) g(h_i X, Q Y)\big] ,\\
\label{E-3.31}
 g((\nabla_{{f} X} {f}){f} Y, Z) &= g((\nabla_X {f})Q Z, Y) + \sum\nolimits_{i}\big[ \eta^i(X) g(h_i Z,QfY) +   \eta^i(Y) g(h_iX, Q fZ) \notag  \\ 
 & \ +\eta^i(Z) g({f} h_i X, \widetilde Q Y)\big] .
\end{align}
\end{Lemma}

\begin{proof}
Differentiating \eqref{2.1} and using \eqref{E-nS-10b}, and the skew-symmetry of $\nabla_X {f}$, we get \eqref{E-3.29}.
%
%Next, 
We obtain \eqref{E-3.30} from \eqref{E-3.29} by the definition of weak nearly ${\cal C}$-manifolds, followed by switching $X$ and $Y$.
%
%Finally, 
Replacing $Y$ by ${f} Y$ in \eqref{E-3.30} and using \eqref{E-3.29} and \eqref{2.1}, we get~\eqref{E-3.31}.
\end{proof}

\begin{Lemma}\label{L-R01}
The curvature tensor of a weak nearly ${\cal C}$-manifold satisfies the equalities
\begin{align}\label{E-3.4}
 &g(R_{{f} X,Y}Z, V) +g(R_{X,{f} Y}Z, V) +g(R_{X,Y}{f} Z, V) +g(R_{X,Y}Z, {f} V) = 0,\\
 \label{E-nS-05ccc}
& R_{\,\xi_j,\, \xi_k}=0\quad(1 \le j,k\le s);
\end{align}
moreover, if the condition  \eqref{E-nS-04c} is true, then
\begin{align}
\label{E-nS-04ccc}
& g(R_{\,\xi_j, Z}\,{f} X,{f} Y) = 0 \quad (1 \le j \le s),
\quad \mbox{\rm hence,}\ \mathcal{D} \ \mbox{\rm is curvature invariant}.
\end{align}
\end{Lemma}

\begin{proof}
Differentiating \eqref{E-nS-01b}, we~find
\begin{eqnarray}\label{EF-nS-01}
 (\nabla^2_{X,Y}\,{f})Z + (\nabla^2_{X,Z}\,{f})Y = 0 \quad (X,Y,Z \in \mathfrak{X}_M).
\end{eqnarray}
Applying the Ricci identity, see \cite{E-2005},
\begin{equation}\label{E-nS-05}
 g((\nabla^2_{X,Y}{f})V, Z) - g((\nabla^2_{Y,X}{f})V, Z) = g(R_{X,Y}{f} V, Z) + g(R_{X,Y}V, {f} Z),
\end{equation}
from \eqref{EF-nS-01} and the skew-symmetry of $\nabla^2_{X,Y}\,{f}$ we get
\begin{eqnarray}\label{E-3.7}
 g(R_{X,Y}Z, {f} V) -g(R_{X,Y}V, {f} Z)  + g((\nabla^2_{X,Z}\,{f}) Y, V) - g((\nabla^2_{Y,Z}\,{f})X, V) = 0.
\end{eqnarray}
By Bianchi and Ricci identities, we find
\begin{align}\label{E-3.8}
\nonumber
 g(R_{X,Y}Z, {f} V) &= -g(R_{Y,Z}X, {f} V) - g(R_{Z,X}Y, {f}V) \\
 &= g((\nabla^2_{Y,Z}\,{f}) V, X) - g((\nabla^2_{Z,Y}\,{f})V, X)
  -g(R_{Y,Z}V, {f} X) - g(R_{Z,X}Y, {f} V).
\end{align}
Substituting \eqref{E-3.8} into \eqref{E-3.7}, we get 
\begin{align}\label{E-3.9}
\nonumber
 & g(R_{X,Z}Y, {f} V) -g(R_{X,Y}V, {f} Z) - g(R_{Y,Z}V, {f} X) \\
 &-g((\nabla^2_{Z,Y}\,{f})V, X) - g((\nabla^2_{X,Z}\,{f})V, Y) 
  = 2\,g((\nabla^2_{Y,Z}\,{f})X, V).
\end{align}
On the other hand, using the Ricci identity \eqref{E-nS-05} again, we see that
\begin{eqnarray}\label{E-3.10}
 g(R_{X,Z}Y, {f} V) -g(R_{X,Z}V, {f} Y)  - g((\nabla^2_{X,Z}\,{f})Y, V) +g((\nabla^2_{Z,X}\,{f})Y, V) = 0.
\end{eqnarray}
Adding \eqref{E-3.10} to \eqref{E-3.9} and applying \eqref{EF-nS-01}, we get
\begin{eqnarray}\label{E-3.11}
 2\,g(R_{X,Z}Y, {f} V) -g(R_{X,Y}V, {f} Z)  - g(R_{Y,Z}V, {f} X) -g(R_{X,Z}V, {f} Y) = 2\,g((\nabla^2_{Y,V}\,{f})Z, X).
\end{eqnarray}
Swapping $Y$ and $V$ in \eqref{E-3.11}, we find
\begin{eqnarray}\label{E-3.12}
 2\,g(R_{X,Z}V, {f} Y) -g(R_{X,V}Y, {f} Z)  - g(R_{V,Z}Y, {f} X) -g(R_{X,Z}Y, {f} V) = 2\,g((\nabla^2_{V,Y}\,{f})Z, X).
\end{eqnarray}
Subtracting \eqref{E-3.12} from \eqref{E-3.11}, and using Bianchi and Ricci identities, we get the equality,
which by replacing $Z$ and $Y$ gives \eqref{E-3.4}.

%\textbf{2}.
Replacing $X$ by ${f} X$ in \eqref{E-3.4} and using \eqref{2.1}, we have
\begin{align}\label{E-3.13}
\nonumber
 -g(R_{Q X,Y}Z, V) & +\sum\nolimits_{i}\eta^i(X)\,g(R_{\,\xi_i,Y}Z, V)  +g(R_{{f} X, {f} Y}Z, V) \\ 
 &+ g(R_{{f} X, Y}{f} Z, V) +g(R_{{f} X, Y}Z, {f} V) = 0.
\end{align}
Exchanging $X$ and $Y$ in \eqref{E-3.13}, we find
\begin{align}\label{E-3.14}
\notag 
 g(R_{X, Q Y}Z, V) & +\sum\nolimits_{i} \eta^i(Y)\,g(R_{\,\xi_i,X}Z, V) - g(R_{{f} X, {f} Y}Z, V)  \\
 & + g(R_{{f} Y, X}{f} Z, V) + g(R_{{f} Y, X}Z, {f} V) = 0.
\end{align}
Subtracting \eqref{E-3.14} from \eqref{E-3.13}, we obtain
\begin{align}\label{E-3.15}
\nonumber
 & 2\,g(R_{{f} X,{f} Y}Z, V) -2\,g(R_{X,Y}Z, V) +\sum\nolimits_{i}[\eta^i(X)\,g(R_{\,\xi_i,Y}Z, V) -\eta^i(Y)\,g(R_{\,\xi_i,X}Z, V)] \\
\nonumber
 & + g(R_{{f} X, Y}{f} Z, V) - g(R_{{f} Y, X}{f} Z, V) + g(R_{{f} X, Y}Z, {f} V) - g(R_{{f} Y, X}Z, {f} V) \\
 & - g(R_{\widetilde Q X, Y}Z, V) - g(R_{X, \widetilde Q Y}Z, V) = 0.
\end{align}
Then, inserting $fZ$ in place of $ Z$ and also $V$ by ${f} V$ in \eqref{E-3.4} and using \eqref{2.1}, we get two equations
\begin{align}\label{E-3.16}
 g(R_{X, Y}Q Z, V) &= \sum\nolimits_{i}\eta^i(Z)\,g(R_{X,Y}\,\xi_i, V) + g(R_{X,Y}{f} Z, {f} V) \notag\\ 
 & \quad + g(R_{X,{f} Y}{f} Z, V) + g(R_{{f} X,Y}{f} Z, V),\\
\label{E-3.17}
%\nonumber
 g(R_{X, Y}Z, Q V) &= \sum\nolimits_{i}\eta^i(V)\,g(R_{X,Y}Z,\,\xi_i) 
 + g(R_{X,Y}{f} Z, {f} V) \notag\\ 
 & \quad + g(R_{{f} X, Y}Z, {f} V) + g(R_{X,{f} Y}Z, {f} V).
\end{align}
Adding \eqref{E-3.16} to \eqref{E-3.17}, and substituting the gotten equation into \eqref{E-3.15}, we have
\begin{eqnarray}\label{E-3.18}
\nonumber
 && 2\,g(R_{{f} X,{f} Y}Z, V) -2\,g(R_{X,Y}{f} Z, {f} V) + \sum\nolimits_{i} \big[\eta^i(X)\,g(R_{\,\xi_i,Y}Z, V)-\eta^i(Z)\,g(R_{X,Y}\,\xi_i, V) \\
 && \quad -\eta^i(V)\,g(R_{X,Y}Z,\,\xi_i)\, -\eta^i(Y)\,g(R_{\,\xi_i,X}Z, V) \big] +\delta(X,Y,Z,V) = 0,
\end{eqnarray}
where
\begin{equation*}
 \delta(X,Y,Z,V) = g(R_{X, Y}\widetilde Q Z, V) 
 +g(R_{X,Y}Z, \widetilde Q V) -g(R_{\widetilde Q X,Y}Z,V) 
 -g(R_{X, \widetilde Q Y}Z, V).
\end{equation*}
Substituting $fX$ in place of $ X$ and also $fY$ in place of $ Y$ in \eqref{E-3.18} and using \eqref{2.1}, we obtain
\begin{align}\label{E-3.19}
\notag
 & 2\,g(R_{Q X, Q Y}Z, V) 
  + 2\,\sum\nolimits_{i}\big[\eta^i(Y)\,g(R_{\,\xi_i,Q X}Z, V)-\eta^i(X)\,g(R_{\,\xi_i, Q Y}Z, V) \big]- 2\,g(R_{{f} X, {f} Y}\,{f} Z, {f} V) \\
 &  +\delta({f} X, {f} Y,Z,V) 
  = \sum\nolimits_{i}\big[\eta^i(Z)\,g(R_{\,\xi_i, V}{f} X,{f} Y) -\eta^i(V)\,g(R_{\,\xi_i, Z}\,{f} X, {f} Y) \big].
\end{align}
Inserting $X=\xi_j$ and $Y=\xi_k$ in \eqref{E-3.19}, we acquire $g(R_{\xi_j, \xi_k}Z, V)=0 $ for all $1 \le j,k\le s$;
hence \eqref{E-nS-05ccc}. 
This also yields $g(R_{\xi_j,\, \xi_k}fZ, fV)=0 $.
Taking $V=\xi_j$ in \eqref{E-3.19} and using the previous relation, we obtain
\begin{align}\label{E-3.20}
\nonumber
 & 2\,g(R_{\,Q X,\,Q Y}Z, \xi_j) +2\sum\nolimits_{i}\big[\eta^i(Y)\,g(R_{\,\xi_i,Q X}Z, \xi_j)-\eta^i(X)\,g(R_{\,\xi_i, Q Y}Z, \xi_j)\big] \\
%\nonumber
 & +\,g(R_{\,\xi_j, Z}\,{f} X, {f} Y) +\delta({f} X, {f} Y,Z, \,\xi_j) = 0\quad (1 \le j \le s).
\end{align}
Replacing $X$ by ${f} X$ and $Y$ by ${f} Y$ in \eqref{E-3.20},
we obtain
\begin{align*}
 & 4\,g(R_{\,Qf X,\,Q fY}Z, \xi_j) -2\sum\nolimits_{i}\big[\eta^i(Y)\,g(R_{\,\xi_i,Q X}Z, \xi_j)-\eta^i(X)\,g(R_{\,\xi_i, Q Y}Z, \xi_j)\big] \\
%\nonumber
 & +2\,g(R_{\,\xi_j, Z}\,{Q} X, {Q} Y)  +2\,\delta({f}^2 X, {f}^2 Y,Z, \,\xi_j) = 0\quad (1 \le j \le s).
\end{align*} 
Adding this to \eqref{E-3.20} and using $Q={\rm Id} +\widetilde Q$, we get 
\begin{align}\label{E-3.22}
\notag
 & 3\,g(R_{\,\xi_j, Z}\,{f} X, {f} Y) =
 -4\,g(R_{\,\xi_j, Z}\,\widetilde Q {f} X,\,\widetilde Q{f} Y) 
 -4\,g(R_{\,\xi_j, Z}\,{f} X,\,\widetilde Q {f} Y) \\
 & -4\,g(R_{\,\xi_j, Z}\,\widetilde Q {f} X,\,{f} Y)
 +\delta({f} X, {f} Y, Z, \,\xi_j) 
 +2\,\delta({f}^2 X, {f}^2 Y, Z,\,\xi_j)
 \quad(1 \le j \le s).
\end{align}
%for all $1 \le j \le s$. 
From \eqref{E-3.22}, using \eqref{E-nS-04c}, \eqref{E-nS-05ccc}
and the equality $\delta({f} X,{f} Y,Z,\xi_j)=0$,
we get \eqref{E-nS-04ccc}. Since ${f}|_{\mathcal{D}}$ is non-degenerate, the distribution $\mathcal{D}$ is curvature invariant, see \eqref{E-nS-04cc}.
\end{proof}
%%%%%%%%%%%%%%%%%%%%%%%%%%%%%%%%%%%%%%%%%
%The following lemma generalizes Lemmas~3.3 in \cite{E-2005}.

\begin{Lemma}\label{L-nS-04}
For a weak nearly ${\cal C}$-manifold 
%$(M^{\,2n+s}, {f},Q, \vec\xi, \vec\eta,g)$ 
satisfying conditions \eqref{E-30b-xi}, \eqref{E-30-xi}, \eqref{E-nS-10} and \eqref{E-nS-04c}, we get
\begin{align}
\label{E-3.23}
 & R_{\,\xi_i,\, X}Y = -(\nabla_X h_i)Y, \\
%%%%%%
\label{E-3.24}
 & (\nabla_X h_i)Y = \sum\nolimits_{j}[g(h_ih_j X, Y)\,\xi_j - \eta^j(Y)\,h_ih_j X] , \\
%%%%%%
\label{E-3.25}
 & {\rm Ric}\,(\xi_i, Z) = -\sum\nolimits_{j}\eta^j(Z)\,\operatorname{tr}(h_i h_j) .
\end{align}
In particular, $\nabla_{\xi_j} h_i=0$ and ${\rm tr}(h_i^2) = const$ for all $1 \le i,j \le s$.
Moreover,
%By \eqref{E-3.24} and \eqref{E-3.23}, we get
\begin{align}\label{E-3.23b}
 g(R_{\,\xi_i, X}Y, Z) &= -g((\nabla_X\,h_i)Y, Z) =\sum\nolimits_{j} [\eta^j(Y)\,g(h_ih_j X, Z) -\eta^j(Z)\,g(h_ih_j X, Y)], \\
%%%%%%%%%%  
\label{E-3.6}
g(R_{{f} X,{f} Y} Z, V) & = g(R_{X, Y}{f} Z, {f} V) - \tfrac12\,\delta(X,Y,Z,V), \\
\label{E-3.5}
\nonumber
g(R_{{f} X,{f} Y}{f} Z, {f} V) & = g(R_{Q X, Q Y}Z, V) + \tfrac12\,\delta({f} X,{f} Y,Z,V) \\
&\quad +\sum\nolimits_{i}\big[\eta^i(Y)\,g(R_{\,\xi_i, Q X}Z, V)-\eta^i(X)\,g(R_{\,\xi_i,Q Y}Z, V)\big] . \end{align}
\end{Lemma}

\begin{proof}
By  Lemma \ref{L-nS-01},  $\xi_i$ is a Killing vector field, and applying \eqref{E-nS-04} in conjunction with  \eqref{E-c-01}, we get \eqref{E-3.23}.
Replacing $Y$ by ${f} Y$ and $Z$ by ${f} Z$ in $g(R_{\,\xi_i, X}Y, Z) = -g((\nabla_X\,h_i)Y, Z)$, see \eqref{E-3.23},
and using \eqref{E-nS-04ccc}, 
we get $g((\nabla_X h_i)\,{f} Y, {f} Z) = 0$,~hence,
\begin{equation}\label{E-nS-03f}
 g((\nabla_X h_i) Y, Z) = 0\quad (Y, Z\in{\cal D}, \ 1 \le i \le s).
\end{equation}
Then, using \eqref{E-nS-03f}, we find the $\ker f$-component and ${\cal D}$-component of $(\nabla_X h_i)Y$:
\begin{eqnarray*}
 && g((\nabla_X h_i)Y,\,\xi_j) = g(\nabla_X(h_i\,Y), \,\xi_j) = -g(h_i\,Y,\,\nabla_X\,\xi_j) = g(h_i h_j X, Y),\\
 && g((\nabla_X\,h_i)Y, Z) = \sum\nolimits_{j}\eta^j(Y)\,g((\nabla_X h_i)\,\xi_j, \,Z)
 = -\sum\nolimits_{j}\eta^j(Y)\,g(h_ih_j X, Z)\quad (Z\in{\cal D}),
\end{eqnarray*}
from which \eqref{E-3.24} follows.
From \eqref{E-3.24} with $X=\xi_k$ we find $\nabla_{\xi_k} h_i=0$.
%, $1 \le k \le s$.
Note that 
for $Y=\xi_j$, \eqref{E-3.24} reduces to \eqref{E-nS-01a}. 
%\textbf{3}.
By \eqref{E-3.24} and \eqref{E-3.23}, we get \eqref{E-3.23b}.
Let $\{e_k\}_{k=1,\ldots,2n+1}$ be a local orthonormal frame on $M$ with $e_{2n+j}=\xi_j\ (1 \le j \le s)$.
Putting $X=Y=e_k$ in \eqref{E-3.24} for $\{e_k\}$, then using \eqref{E-nS-04cc} and summing over $k=1,\ldots,2n+s$, we get
\begin{equation*}
 \sum\nolimits_{k=1}^{2n+s}(\nabla_{e_k} h_i)\, e_k
=\sum\nolimits_{j} \operatorname{tr}(h_i h_j)  \xi_j .
\end{equation*}
Next, taking the inner product of \eqref{E-3.23} with $Z$, tracing over $X$ and $Y$, and then substituting the above relation, we obtain \eqref{E-3.25}.
Replacing $Y$ by $h_i Y$ in \eqref{E-3.24}, putting $Y=e_k$ in the gotten equation and summing over $k=1,\ldots, 2n+s$,
we get
\[
\operatorname{tr}((\nabla_X h_i)h_i)=\sum\nolimits_{k=1}^{2n+s}g((\nabla_X h_i)h_ie_k,e_k) = \sum\nolimits_{j}\sum\nolimits_{k=1}^{2n+s} \eta^j(e_k)\,g(h_ih_jX,h_i e_k) =0.
\]
This implies $X({\rm tr}(h_i^2)) = 0\ (X\in\mathfrak{X}_M)$, i.e., ${\rm tr}(h_i^2) = const$.
Finally, using \eqref{E-3.18} and \eqref{E-3.23b}, we obtain 
\begin{align*}
 & 2 g(R_{{f} X,{f} Y}Z, V) -2 g(R_{X,Y}{f} Z, {f} V) 
 +\!\sum\nolimits_{i,j=1}^s \!\big[\eta^i(X)\eta^j(Z)g(h_ih_jY,V) 
 -\eta^i(Z)\eta^j(X)g(h_ih_jV,Y)   \\
 & -\eta^i(X)\eta^j(V)g(h_ih_jY,Z)
 +\eta^i(V)\eta^j(X)g(h_ih_jZ,Y) 
 +\eta^i(Z)\eta^j(Y)g(h_ih_jV,X) \\ 
 & -\eta^i(Y)\eta^j(Z)g(h_ih_jX,V)
 {+}\eta^i(Y)\eta^j(V)g(h_ih_jX,Z) 
 {-}\eta^i(V)\eta^j(Y)g(h_ih_jZ,X)\big] +\delta(X,Y,Z,V) =0.
\end{align*}
Simplification of this using the skew-symmetry of $h_i$ yields \eqref{E-3.6}.
%\textbf{4}.
Switching $X$ by ${f} X$ and $Y$ by ${f} Y$ in \eqref{E-3.6} and using \eqref{2.1}, we get \eqref{E-3.5}.
\end{proof}
%%%%%%%%%%%%%%%%%%%%%%%

\begin{Remark}\label{Rem-delta}\rm
The function $\delta$ of a weak nearly ${\cal C}$-manifold has the following symmetries:
\begin{eqnarray*}
 \delta(Y,X, Z,V) = \delta(X,Y, V,Z) = \delta(Z,V, X,Y) = -\delta(X,Y, Z,V) .
\end{eqnarray*}
If \eqref{E-nS-04c} is true, then by \eqref{E-3.23b}, we acquire
\[
 \delta(\xi_i,Y,Z,V)\!=\delta(X,\xi_i,Z, V)\!=\delta(X,Y,\xi_i, V)\!=\delta(X,Y,Z, \xi_i)=0\quad(1 \le i \le s).
\] 
\end{Remark}

\begin{Lemma}%\label{L-R03}
For a weak nearly ${\cal C}$-manifold satisfying \eqref{E-30b-xi}, \eqref{E-30-xi}, \eqref{E-nS-10} and \eqref{E-nS-04c}, 
we have
for $1 \le j \le s$
%the covariant derivative of $f$ satisfies the following:  
\begin{align}\label{E-3.50}
 g((\nabla_{X} {f})Y, Q{f} h_j Z) {=} \sum\nolimits_{i} \big[\eta^i(X)g(h_j Y,h_iZ){-}\eta^i(Y)g(h_jX,h_iZ){+} \eta^i(X)g(\widetilde Q(Q{+} {\rm Id})Y,h_ih_jZ)\big]. 
\end{align}
\end{Lemma}

\begin{proof} It consists of two parts.
In Part 1, our formulas depend on four vectors %$X,Y,Z,V\in\mathfrak{X}_M$ 
from $TM$ and contain many additional $\widetilde Q$-dependent terms.
In~Part 2, the~formulas depend on three vectors from $TM$ and contain a few $\widetilde Q$-dependent terms.
%\textbf{Step 1}.

\textbf{Part 1}. Differentiating \eqref{E-3.29} and using
$g((\nabla_X{f})(\nabla_V{f})Y, Z)$ $=-g((\nabla_V{f})Y, (\nabla_X{f})Z)$ gives
\begin{align}\label{E-3.32}
\nonumber
 & g((\nabla_V\,{f})Y, (\nabla_X {f})Z) + g((\nabla_X {f})Y, (\nabla_V\,{f})Z)
 = g((\nabla^2_{V,X}\,{f}){f} Y, Z) + g((\nabla^2_{V,X}\,{f}){f} Z, Y)\\
\nonumber
 & 
 -\sum\nolimits_{i} \big[g(h_iV,Z)\,g(h_iX,Y) +\eta^i(Z)\,g((\nabla_V h_i)X,Y)  +g(h_i V,Y)\,g(h_i X,Z)  \\
 &+\eta^i(Y)\,g((\nabla_V h_i)X,Z)  + \eta^i(Y)g((\nabla_V h_i)X, \widetilde QZ) +\eta^i(Z)g((\nabla_V h_i)X, \widetilde QY) \notag \\ 
 & +g(h_iV,Y)g(h_iX,\widetilde Q Z)+ g(h_iV,Z)g(h_iX,\widetilde Q Y)\big]+2 \sum\nolimits_{i,j} \eta^i(Y) \eta^j(Z) g(h_iX, \widetilde Q h_jV).   
\end{align}
%On the other hand,
Using \eqref{2.1}, \eqref{EF-nS-01} and  \eqref{E-3.11}, we find $\nabla^2$-terms in \eqref{E-3.32}:
\begin{align}\label{ER-nS-04aa}
\nonumber
 & g((\nabla^2_{V,X}\,{f}){f} Z, Y) = g((\nabla^2_{V, {f} Z}\,{f})Y, X)
  = -g(R_{X,Y}V, Q Z) +\sum\nolimits_{i}\eta^i(Z)\,g(R_{X,Y}V, \xi_i) \\
 &\qquad -\tfrac12\,g(R_{X,V}{f} Z, {f} Y) - \tfrac12\,g(R_{V,Y}{f} Z, {f} X) - \tfrac12\,g(R_{X,Y}{f} Z, {f} V),\\
%%%%%%%%%%%%%%%
 \label{ER-nS-04a}
\nonumber
 & g((\nabla^2_{V,X}\,{f}){f} Y, Z) = g((\nabla^2_{V, {f} Y}\,{f})Z, X)
  = -g(R_{X,Z}V, Q Y) +\sum\nolimits_{i}\eta^i(Y)\,g(R_{X,Z}V, \xi_i) \\
 &\qquad -\tfrac12\,g(R_{X,V}{f} Y, {f} Z) - \tfrac12\,g(R_{V,Z}{f} Y, {f} X) - \tfrac12\,g(R_{X,Z}{f} Y, {f} V).
\end{align}
%Using \eqref{E-3.23b}, 
Substituting  \eqref{ER-nS-04aa}--\eqref{ER-nS-04a} in  \eqref{E-3.32}, we get
\begin{align}\label{ER-nS-03b}
\nonumber
 & g((\nabla_V\,{f})Y, (\nabla_X {f})Z) + g((\nabla_X {f})Y, (\nabla_V\,{f})Z) = -g(R_{X,Z}V, Q Y) - g(R_{X,Y}V, Q Z) \\
\nonumber
 & - \tfrac12\,g(R_{V,Z}{f} Y, {f} X) - \tfrac12\,g(R_{X,Z}{f} Y, {f} V) - \tfrac12\,g(R_{V,Y}{f} Z, {f} X) - \tfrac12\,g(R_{X,Y}{f} Z, {f} V)\\
\nonumber
 &   \nonumber
 + \sum\nolimits_{i}\big[\eta^i(Y)\,g(R_{X,Z}V, \xi_i) 
 + \eta^i(Z)\,g(R_{X,Y}V, \xi_i) - g(h_iV,Z)\,g(h_iX,Y) - \eta^i(Z)\,g((\nabla_V h_i)X,Y)\\
 &  - g(h_i V,Y)\,g(h_i X,Z) - \eta^i(Y)\,g((\nabla_V h_i)X,Z) - \eta^i(Y)g((\nabla_V h_i)X, \widetilde QZ) -\eta^i(Z)g((\nabla_V h_i)X, \widetilde QY) \notag \\ 
  & -g(h_iV,Y)g(h_iX,\widetilde Q Z)- g(h_iV,Z)g(h_iX,\widetilde Q Y)\big]+2 \sum\nolimits_{i,j} \eta^i(Y) \eta^j(Z) g(h_iX, \widetilde Q h_jV). 
\end{align}
Applying \eqref{E-3.6} twice followed by \eqref{E-3.24} in \eqref{ER-nS-03b}, the curvature terms involving $\xi_i$ vanish, and we~get
%then replacing $(\nabla_V h)X$ by \eqref{E-3.24}, we get
\begin{align}\label{E-3.34}
\nonumber
 & g((\nabla_X {f})Z, (\nabla_V {f})Y) + g((\nabla_X {f})Y, (\nabla_V {f})Z) + g(R_{X,Z}V, Q Y) + g(R_{X,Y}V, Q Z)  \\
\nonumber
 & - g(R_{V,Z}{f} X, {f} Y) - g(R_{X,Z}{f} V, {f} Y) 
 + \sum\nolimits_{i} \big[g(h_i V,Y)\,g(h_i X,Z) + g(h_iV,Z)\,g(h_iX,Y)\big] \\
 = &-\tfrac14\delta(X,Z,V,Y) {-} \tfrac14\delta(V,Z,X,Y) {-}\sum\nolimits_{i} [g(h_iV,Y)g(h_iX,\widetilde Q Z)+ g(h_iV,Z)g(h_iX,\widetilde Q Y)\big] \notag \\ 
 & + \sum\nolimits_{i,j} \big[2\eta^i(Y) \eta^j(Z) g(h_iX, \widetilde Q h_jV){-}\eta^i(Y)\eta^j(X)g(h_ih_jV, \widetilde QZ) {-}\eta^i(Z)\eta^j(X)g(h_ih_jV, \widetilde QY)\big]. 
\end{align}
Replacing  $Z$ by ${f} Z$ and $V$ by ${f} V$ in \eqref{E-3.34}, we find
\begin{align}\label{E-3.35}
\nonumber
 & g((\nabla_X {f}){f} Z, (\nabla_{{f} V}\,{f})Y) 
 + g((\nabla_X {f})Y, (\nabla_{{f} V}\,{f}){f} Z)
 + g(R_{X,{f} Z}{f} V, Q Y) {-} g(R_{X,{f} Z}{f}^2 V, {f} Y) \\
\nonumber
 & + g(R_{X,Y}{f} V,{f} Q Z) - g(R_{{f} V,{f} Z}{f} X,{f} Y) +\sum\nolimits_{i}\big[ g(h_iX,Y)\,g(h_i{f} V,{f} Z)+ g(h_i X,{f} Z)\,g(h_i{f} V,Y)\big]  \\
%\nonumber
 & =- \tfrac14\,\delta(X,{f} Z,{f} V,Y) - \tfrac14\,\delta({f} V,{f} Z,X,Y) -\sum\nolimits_{i,j} \eta^i(Y)\eta^j(X)g(h_ih_jfV, \widetilde QfZ)  \notag \\ 
 & -\sum\nolimits_{i} [g(h_ifV,Y)g(h_iX,\widetilde Qf Z)+ g(h_ifV,fZ)g(h_iX,\widetilde Q Y)\big].
\end{align}
Using \eqref{2.1}, \eqref{E-3.29}, \eqref{E-3.30}, \eqref{E-3.31} and Lemma~\ref{L-nS-02}, we find
\begin{align}\label{E-3.36}
\nonumber
&g((\nabla_X {f}){f} Z, (\nabla_{{f} V}\,{f})Y) = g(Q(\nabla_X {f})Z, (\nabla_{V}\,{f})Y) -\sum\nolimits_{i,j} \eta^i(Z)\,\eta^j(V) g(Qh_i X,Q h_j Y) \\
\nonumber
&\quad +\sum\nolimits_{i} \big[ \eta^i(Z) g(Q{f} h_i X, (\nabla_V\,{f})Y)  - g(X, {f} h_i Z)\,g(V, {f} h_i Y) - \eta^i(V)\,g((\nabla_X {f})Z, Q{f} h_i Y)
\\ &\quad   + g(h_i X, Q Z) g(h_i V, Q Y)\big],\\ %%%%%%%%%%%%%%%%
\label{E-3.37}
\nonumber
& g((\nabla_X {f})Y, (\nabla_{{f} V}{f}){f} Z) = -g(Q(\nabla_X {f})Y, (\nabla_{V}\,{f})Z) +\sum\nolimits_{i} \big[\eta^i(V) g(Q{f} h_i Z, (\nabla_{X} {f})Y) \\
&\quad  -\eta^i(Z) g(Q{f} h_i V, (\nabla_{X} {f})Y)
 +g({f} h_i X, Y) g({f} h_i V, \widetilde Q Z )\big].
\end{align} 
%%%%%%%%%%%%%%%%%%%%%%%%%%%%%%%%%%%%%%%%%
From  \eqref{E-3.23b} and Lemma~\ref{L-nS-02}, we have
\begin{align}\label{E-3.38}
\nonumber
 g(R_{X, {f} Z}{f} V, Y) - g(R_{X, {f} Z}{f}^2 V, {f} Y) &= g(R_{X, {f} Z}{f} V, Y) +g(R_{X, {f} Z} Q V, {f} Y) \\
 & \quad  -\sum\nolimits_{i,j}\eta^i(X)\,\eta^j(V)\, 
 g(h_jh_i Z, QY).
\end{align} 
On the other hand, from \eqref{E-3.4}, \eqref{E-3.6} and \eqref{2.1} it follows that
\begin{align}
\label{E-3.39}
& g(R_{X, Z}{f} V, {f} Y) + g(R_{X, Z}{f}^2 V, Y) + g(R_{X, {f} Z}{f} V, Y) + g(R_{{f} X, Z}{f} V, Y) = 0 ,\\
\label{E-3.40}
\nonumber
-&g(R_{{f} X,Z}{f} V, Q Y) +\sum\nolimits_{i}\eta^i(Y)\,g(R_{{f} X, Z}{f} V,\xi_i) =-g(R_{Q X,{f} Z}V,{f} Y) \\
& \quad +\sum\nolimits_{i}\eta^i(X)\,g(R_{\xi_i\,,{f} Z}V,{f} Y) +\tfrac12\,\delta({f} X,Z,V,{f} Y).
\end{align} 
Summing up the formulas \eqref{E-3.39} and \eqref{E-3.40}
(and using \eqref{2.1}, \eqref{E-3.23b} and Lemma~\ref{L-nS-02}), we obtain
\begin{align}\label{E-3.41}
\nonumber
 g(R_{X, {f} Z}{f} V, Y) &+ g(R_{Q X,{f} Z}V,{f} Y) = g(R_{X, Z} Q V, Y) - g(R_{X, Z}{f} V, {f} Y) 
 + g(R_{{f} X,Z}{f} V, \widetilde Q Y) \\
\nonumber
 &+\sum\nolimits_{i,j}\big[ \eta^i(X)\,\eta^j(V)\,g(h_ih_j Y, \widetilde Q Z)-\eta^i(Y)\,\eta^j(Z)\,g(h_ih_j V, Q X) \\
 & + \eta^i(Z)\,\eta^j(V)\,g(h_ih_j Y, X)\big]
  + \tfrac12\,\delta({f} X,Z,V,{f} Y).
\end{align}
 Substituting \eqref{E-3.41} into \eqref{E-3.38}, we get
\begin{align}\label{E-3.42}
\nonumber
 & g(R_{X, {f} Z}\,{f} V, Y) - g(R_{X, {f} Z}\,{f}^2 V, {f} Y) = g(R_{X, Z}\,Q V, Y) - g(R_{X, Z}{f} V, {f} Y) \\
\nonumber
 & + \sum\nolimits_{i,j=1}^s\big[\eta^i(Z)\eta^j(V)\,g(h_ih_jY, X)  - \eta^i(Y)\eta^j(Z)\,g(h_jh_i X, Q V)  - \eta^i(X)\,\eta^j(V)\,g(h_jh_iZ, Y)\big] \\
 & - g(R_{\widetilde Q X, {f} Z} V, {f} Y) + g(R_{X, {f} Z} \widetilde Q V, {f} Y) + g(R_{{f} X, Z}{f} V, \widetilde Q Y) +\tfrac12\,\delta({f} X,Z,V,{f} Y) .
\end{align} 
Using \eqref{E-3.23b} and \eqref{E-3.6}--\eqref{E-3.5}, we have
\begin{align}\label{E-3.43}
\nonumber
 & g(R_{X, Y}{f} V, {f} Z) - g(R_{{f} V, {f} Z}{f} X, {f} Y) =  g(R_{V, Z}{f} X,{f} Y)  -g(R_{V, Z}Q X, Q Y) \\
\nonumber
 & +\sum\nolimits_{i,j}\big[\eta^i(X)\,\eta^j(V)\,g(h_ih_j Q Y, Z)-\eta^i(X)\eta^j(Z)\,g(h_ih_j Q Y, V) +\,\eta^i(Y)\,\eta^j(Z)\,g(h_ih_j Q X, V) \notag  \\
 & -\eta^i(Y)\,\eta^j(V)\,g(h_ih_j Q X, Z)\big]
 -\tfrac12\,\delta({f} X,{f} Y,Z,V) - \tfrac12\,\delta(X,Y,Z,V) .
\end{align} 
Applying \eqref{E-3.36}, \eqref{E-3.37}, \eqref{E-3.42} and \eqref{E-3.43} in \eqref{E-3.35},
and using Lemma~\ref{L-nS-02}, we obtain
\begin{align}\label{E-3.44}
\nonumber
 & g(Q(\nabla_X {f})Z, (\nabla_{V} {f})Y) - g(Q(\nabla_X {f})Y, (\nabla_{V} {f})Z) + \sum\nolimits_{i}\big[\eta^i(V) g((\nabla_{X} {f})Y, Q{f} h_i Z)   \\
\nonumber
 & - \eta^i(V)\,g((\nabla_X {f})Z, Q{f} h_i Y)
 +\eta^i(Z) g((\nabla_V {f})Y, Q{f} h_i X)  - \eta^i(Z) g((\nabla_{X} {f})Y, Q{f} h_i V) \\
\nonumber
 & + g(h_i X, Q Z) g(h_i V, Q Y) - g(h_iX,Y)\,g(h_i V, Q Z) \big]  + \sum\nolimits_{i,j}\big[ 2\,\eta^i(Z)\eta^j(V) g(h_ih_jY,X)\\
\nonumber
 & - \eta^i(X)\,\eta^j(Z)\,g(h_ih_j Y, QV) - \eta^i(Y)\,\eta^j(V)\,g(h_ih_j X, QZ)\big ] + g(R_{X, Z}\,Q V, Y) - g(R_{X, Z}{f} V, {f} Y) \\
\nonumber
 &  
 + g(R_{V, Z}{f} X,{f} Y) 
 - g(R_{V, Z}\,Q X, Q Y) = g(R_{\widetilde Q X, {f} Z} V, {f} Y) - g(R_{X, {f} Z} \widetilde Q V, {f} Y) \\
\nonumber
&- g(R_{X,{f} Z}{f} V, \widetilde Q Y)
 - g(R_{X,Y}{f} V, {f} \widetilde Q Z) - g(R_{{f} X, Z}{f} V, \widetilde Q Y)-\sum\nolimits_{i} [g({f} h_i X, Y) g({f} h_i Z, \widetilde Q V) \\ 
&+ g(h_ifV,Y)g(h_iX,\widetilde Qf Z)+ g(h_ifV,fZ)g(h_iX,\widetilde Q Y)\big] + \sum\nolimits_{i,j}^s [ \eta^i(Z)\eta^j(V)g(\widetilde Q(Q{+} {\rm Id})h_iX,h_jY)\notag  \\
\nonumber
&  -\eta^i(X)\,\eta^j(V)\,g(h_ih_j Y, \widetilde QZ) - \eta^i(Y)\eta^j(X)g(h_ih_jfV, \widetilde QfZ) ] 
+ \tfrac12\,\delta({f} X,{f} Y,Z,V) \notag \\  
& - \tfrac12\,\delta({f} X,Z,V,{f} Y) + \tfrac12\,\delta(X,Y,Z,V)
-\tfrac14\,\delta(X,{f} Z,{f} V,Y) - \tfrac14\,\delta({f} V,{f} Z,X,Y).
\end{align}
%{\color{red} In (62) and after we can use Lemma 3.6: $h_ih_j = h_jh_i$.}
Adding \eqref{E-3.44} to \eqref{E-3.34}, we obtain
\begin{align}\label{E-3.45}
\nonumber
 & 2\,g((\nabla_X {f})Z, (\nabla_V {f})Y) + g(\widetilde Q(\nabla_X {f})Z, (\nabla_{V} {f})Y)
 - g(\widetilde Q(\nabla_X {f})Y, (\nabla_{V} {f})Z) \\
\nonumber
 & + \sum\nolimits_{i}\big[\eta^i(V) g((\nabla_{X} {f})Y, Q{f} h_i Z) - \eta^i(V)\,g((\nabla_X {f})Z, Q{f} h_i Y) +\eta^i(Z) g((\nabla_V {f})Y, Q{f} h_i X) \\
\nonumber
 & - \eta^i(Z) g((\nabla_{X} {f})Y, Q{f} h_i V) + g(h_i X, Q Z) g(h_i V, Q Y) + g(h_i V,Y)\,g(h_i X,Z)\big] \\
%%%%%
\nonumber
 & + \sum\nolimits_{i,j}^s \big[2\,\eta^i(Z)\eta^j(V) g(h_ih_j Y,X) - \eta^i(X)\eta^j(Z)g(h_ih_j Y, QV) - \eta^i(Y)\eta^j(V)g(h_ih_j X, QZ)\big] \\
\nonumber
 & + 2\,g(R_{X, Z}\,V, Y) - 2\,g(R_{X,Z}{f} V, {f} Y) - g(R_{V, Z}\,Q X, Q Y) + g(R_{X,Y}V, Q Z) \\
%%%%%%%%%%%%%%%%%%%
\nonumber
& = g(R_{\widetilde Q X, {f} Z} V, {f} Y) {-} g(R_{X, {f} Z} \widetilde Q V, {f} Y) {-} g(R_{X,{f} Z}{f} V, \widetilde Q Y) - g(R_{X, Z}\,\widetilde Q V, Y) -g(R_{X,Z}V, \widetilde Q Y)\\
\nonumber
& -g(R_{X,Y}{f} V, {f}\widetilde Q Z) -g(R_{{f} X, Z}{f} V,\widetilde Q Y) -\sum\nolimits_{i,j}  [\eta^i(X)\,\eta^j(V)\,g(h_ih_j Y, \widetilde QZ) \notag\\ 
\notag& 
- \eta^i(Z)\eta^j(V)g(\widetilde Q(Q{+} {\rm Id})h_iX,h_jY)-2\eta^i(Y) \eta^j(Z) g(h_iX, \widetilde Q h_jV)+\eta^i(Y)\eta^j(X)g(h_ih_jV, \widetilde QZ) \\ 
& +\eta^i(Z)\eta^j(X)g(h_ih_jV, \widetilde QY) +\eta^i(Y)\eta^j(X)g(h_ih_jfV, \widetilde QfZ)] +\sum\nolimits_{i}\big[g(h_iX,Y)\,g(h_i V,\widetilde Q Z) \notag\\
 \nonumber 
 &  - g({f} h_i X, Y) g({f} h_i Z, \widetilde Q V) {-}g(h_ifV,Y)g(h_iX,\widetilde Qf Z){-} g(h_ifV,fZ)g(h_iX,\widetilde Q Y) {-}g(h_iV,Y)g(h_iX,\widetilde Q Z)\\ 
 &- g(h_iV,Z)g(h_iX,\widetilde Q Y) ]  -\tfrac12\,\delta({f} X,Z,V,{f} Y)   + \tfrac12\,\delta({f} X,{f} Y,Z,V) 
 +\tfrac12\,\delta(X,Y,Z,V)   \nonumber\\ 
& -\tfrac14\delta(X,Z,V,Y) - \tfrac14\delta(V,Z,X,Y) 
-\tfrac14\,\delta(X,{f} Z,{f} V,Y) - \tfrac14\,\delta({f} V,{f} Z,X,Y).
\end{align}
Swapping $X\leftrightarrow Z$ and $V\leftrightarrow Y$ in \eqref{E-3.45},
then subtracting the gotten equation from \eqref{E-3.45} and using \eqref{E-nS-01b}, we~get
\begin{align}\label{E-3.46}
\nonumber
 & \sum\nolimits_{i}\big[\eta^i(V) g((\nabla_{X} {f})Y, Q{f} h_i Z) - \eta^i(V)\,g((\nabla_X {f})Z, Q{f} h_i Y) +\eta^i(Z) g((\nabla_V\,{f})Y, Q{f} h_i X) \\
\nonumber
 & -\eta^i(Y)\,g((\nabla_Z\,{f})V, Q{f} h_i X)- \eta^i(Z) g((\nabla_{X} {f})Y, Q{f} h_i V) + \eta^i(Y) g((\nabla_{Z} {f})X, Q{f} h_i V)  \\
\nonumber
 & -\eta^i(X) g((\nabla_Y\,{f})V, Q{f} h_i Z) + \eta^i(X) g((\nabla_{Z} {f})V, Q{f} h_i Y)\big]  + g(R_{Z,V}\,Q X, Q Y) + g(R_{Z,V}\,Q X, Y) \\
%\nonumber
%%%%%%%%%%
\nonumber
 & - g(R_{X,Y}\,Q Z, V) {-} g(R_{X,Y}\,Q Z, Q V) {+}2\sum\nolimits_{i,j}\big[ \eta^i(Z)\eta^j(V) g(h_ih_jY, X) 
 {-}\eta^i(X)\eta^j(Y) g(h_ih_j V,Z) ] \\
\nonumber
%%%%%%%%%%%%%%%%%%%
& = g(R_{\widetilde Q X, {f} Z} V, {f} Y) {-} g(R_{X,{f} Z}{f} V, \widetilde Q Y) + g(R_{Z,V}{f} Y, {f}\widetilde Q X)  
{-} g(R_{X,Y}{f} V, {f}\widetilde Q Z) {-} g(R_{\widetilde Q Z, {f} X}Y, {f}V)\\
%%%%%%%%%%%%%%
 \nonumber
& + g(R_{Z,{f} X}{f} Y, \widetilde Q V)  + \sum\nolimits_{i,j}\big[ \eta^i(X)\,\eta^j(V)\,g(h_ih_j Z, \widetilde QY)-\eta^i(Z)\,\eta^j(Y)\,g(h_ih_j V, \widetilde QX) \big] \\
% \nonumber 
 &  + \sum\nolimits_{i}\big[g(h_iX,Y)\,g(h_i V,\widetilde Q Z)- g({f} h_i X, Y) g({f} h_i Z, \widetilde Q V) - g(h_iV,QZ)g(h_iX,\widetilde Q Y)\nonumber\\ 
 &  - g(h_iV,Z)g(h_iX,\widetilde Q Y) {-}g(h_iZ,V)\,g(h_i Y,\widetilde Q X) {+}g({f} h_i Z, V) g({f} h_i X, \widetilde Q Y) {+}g(h_iY,QX)g(h_iZ,\widetilde Q V) \notag  \\ 
 &- g(h_iY,X)g(h_iZ,\widetilde Q V) \big]
  +\tfrac14\delta({f} X,Z,{f} Y,V) + \tfrac14\,\delta({f} X,{f} Y,Z,V) +\tfrac12\,\delta(X,Y,Z,V)\notag \\
 & +\tfrac14\,\delta(X,{f} Z,{f} V,Y) -\tfrac14\,\delta({f} Z,{f} V, X,Y) .
\end{align}
\textbf{Part 2}.
Inserting $\xi_k$ in place of $V$ in \eqref{E-3.46},
then using Lemma~\ref{L-nS-02}, \eqref{E-nS-04cc}, Remark~\ref{Rem-delta} and \eqref{E-nS-01b}, we get
\begin{align*} 
 & g((\nabla_{X} {f})Y, Q{f} h_k Z) = g((\nabla_X {f})Z,Q{f} h_k Y) {-} \sum\nolimits_{i}\big[\eta^i(Z)g(Qh_kY,Qh_iX) 
 {+}\eta^i(Y)g(Qh_kZ,Qh_iX) \\ 
 & +2\eta^i(Z)g(h_kY,h_iX)  {-}\eta^i(X)g(h_ih_kY,\widetilde QZ) \big] 
 {+}g(R_{\xi_k, Z} QX, QY) {+} g(R_{\xi_k, Z} QX, Y) {-}2g(R_{\xi_k, QZ} X,Y).
\end{align*}
Simplifying the above equation using \eqref{2.1} and \eqref{E-3.23b}, we have 
\begin{align}\label{E-3.47}
 g((\nabla_{X} {f})Y, Q{f} h_k Z) &= g((\nabla_X {f})Z,Q{f} h_k Y) {+} \sum\nolimits_{i}\big[\eta^i(Z)g(Qh_kY,Qh_iX) 
 {-}\eta^i(Y)g(Qh_kZ,Qh_iY) \notag \\ 
 & \quad+2\,\eta^i(Z)g(\widetilde Q(Q+ {\rm Id})X,h_ih_kY)   \big] .
\end{align}
By swapping $X$ and $Y$ in \eqref{E-3.47} and adding this to \eqref{E-3.47}, we get
\begin{align}\label{E-3.47b}
 & g((\nabla_Y f)Z,Qfh_kX) +g((\nabla_X f)Z,Qfh_kY)= \sum\nolimits_{i}\big[ \eta^i(X) g(Qh_iY,Qh_kZ) +\eta^i(Y)g(Qh_iX,Qh_kZ )   \notag \\ 
 & -2\eta^i(Z)g(Qh_i X,Qh_kY)+ 4\eta^i(X)g(\widetilde Q(Q+ {\rm Id})Y,h_ih_kZ)\big].
\end{align}
By swapping $Z$ and $X$ in \eqref{E-3.47b}, and adding this to \eqref{E-3.47}, we get
\begin{align*}
 &g((\nabla_X f)Y,Qfh_kZ) = \sum\nolimits_{i} \big[\eta^i(X)g(Qh_i Z,Qh_kY)-\eta^i(Z)g(Qh_kY,Qh_iX) \notag \\
 & \quad +2 \eta^i(X)g(\widetilde Q(Q+ {\rm Id})X,h_ih_kZ)-\eta^i(Z)g(\widetilde Q(Q+ {\rm Id})X,h_ih_kY) \big].
\end{align*}
Finally, simplifying the above equation using the definition of $\widetilde Q$, we get the desired \eqref{E-3.50}.
\end{proof}

%%%%%%%%%%%%%%%%%%%%%%%%%%%%%

Next, we generalize Proposition 3.2 in \cite{NDY-2018} to the case $s>1$.

%Define the spaces  and $\Lambda^k T_p^*M$.

\begin{Lemma}\label{T-new}
Let $\Omega^k(M)$ be the space of 
%smooth 
$k$-forms on an $s$-contact manifold 
$(M^{2n+s},\vec{\eta},{\mathcal R}\oplus{\cal D})$.
If~$n\ge 3$, then the map
 $\mathcal{E}_{d\eta^i}:\ \Omega^2(M)\longrightarrow\Omega^4(M),\, \beta\mapsto d\eta^i \wedge\beta$ is injective for all $i=1,\ldots,s$.
\end{Lemma}

\begin{proof} We denote by $\Omega^k(\mathcal D)$
the space of $k$-forms along $\mathcal D$, i.e., $k$-forms vanishing whenever at least one argument lies in~${\mathcal R}$. The proof of the proposition consists of four steps.

\smallskip\noindent
\underline{Step 1} (Symplecticity on $\mathcal D$ and Lefschetz injectivity):
By the definition of an $s$-contact manifold, $d\eta^i|_{\mathcal D}$ is nondegenerate at every point $p\in M$. 
Each fibre $(\mathcal D_p, d\eta^i|_{{\mathcal D}_p})$
is a $2n$-dimensional symplectic vector space, hence $(d\eta^i|_{{\mathcal D}})^n$ is nowhere vanishing. 
In particular, by \cite[Proposition~1.1]{BPD-2003} (the algebraic Lefschetz property), the linear maps
$\mathcal E_{d\eta^i|_{\mathcal D}}:\,\Omega^2(\mathcal D)\longrightarrow\Omega^4(\mathcal D)$,
$\mathcal E_{d\eta^i|_{{\mathcal D}}}:\,\Omega^1(\mathcal D)\longrightarrow\Omega^3(\mathcal D)$
are injective for $n\ge 3$ and $n\ge 2$, respectively (the latter injectivity, in particular, holds when $n\ge3$).

\smallskip\noindent
\underline{Step 2} (Bundle-level decomposition by $\eta$-degree):
At each $p\in M$ we have a direct-sum decomposition
$T_p^*M={\mathrm{span}\{\eta^1_p,\dots,\eta^s_p\}}\ \oplus\ \mathcal D_p^*,$
which induces a decomposition of exterior powers:
%\[
$\Lambda^k(T_p^*M)\;=\;\bigoplus\nolimits_{j=0}^{\min(k,s)} \Lambda^j (\mathrm{span}\{\eta^1_p,\dots,\eta^s_p\}) \,\wedge\, \Lambda^{k-j}\mathcal D_p^*$.
%\]
Since the ranks are constant and the summands depend smoothly on~$p$, this yields a decomposition of bundles and spaces of sections:
\[
\Omega^k(M) \;=\; \bigoplus\nolimits_{j=0}^{\min(k,s)} {\Lambda^j\!\big(\mathrm{span}\{\eta^1,\dots,\eta^s\}\big)}
\wedge \Omega^{k-j}(\mathcal D).
\]
In particular, every $\beta\in\Omega^2(M)$ admits a unique splitting
\begin{equation}\label{eq:beta-splitting}
\beta \;=\; \beta_0 \;+\; \sum\nolimits_{a=1}^s \eta^a\wedge\mu_a \;+\; \sum\nolimits_{1\le a<b\le s} f_{ab}\,\eta^a\wedge\eta^b,
\end{equation}
with $\beta_0\in\Omega^2(\mathcal D)$, $\mu_a\in\Omega^1(\mathcal D)$ and $f_{ab}\in C^\infty(M)$.

\smallskip\noindent
\underline{Step 3} ($d\eta^i$ has no $\eta$-components):
Expand $d\eta^i$ according to the same decomposition:
\[
d\eta^i \;=\; 
d\eta^i|_{\mathcal D}
%d\eta^i_0 
+ \sum\nolimits_{a=1}^s \eta^a\wedge \gamma_a \;+\; \sum\nolimits_{1\le a<b\le s} h_{ab}\,\eta^a\wedge\eta^b,
\]
with $\gamma_a\in\Omega^1(\mathcal D)$ and $h_{ab}\in C^\infty(M)$. Let $\{\xi_1,\dots,\xi_s\}$ be a local frame  of ${\mathcal R}$ dual to $\{\eta^1,\dots,\eta^s\}$. Using $\ker\,d\eta^i={\mathcal R}$, we have $\iota_{\xi_a}d\eta^i=0$ for each $i$ and $a$. Contracting the above expansion yields
\[
0=\iota_{\xi_a}d\eta^i
=\ \gamma_a \;+\; \sum\nolimits_{b} h_{ab}\,\eta^b \;-\; \sum\nolimits_{b} h_{ba}\,\eta^b
=\ \gamma_a \;+\; 2\sum\nolimits_{b} h_{ab}\,\eta^b .
\]
Since $\{\eta^b\}$ are pointwise independent, it follows that $\gamma_a=0$ and $h_{ab}=0$ for all $a,b$. 
Hence
$d\eta^i = d\eta^i|_{\mathcal D}$,
%, i.e.\ $d\eta^i$ has pure $\eta$-degree 0.
%{\color{red} This phrase is not clear.} 
that means that $d\eta^i$ has no terms involving $\eta^i$.

\smallskip\noindent
\underline{Step 4} (Wedge and compare $\eta$-degrees):
Assume $d\eta^i\wedge\beta=0$. Using \eqref{eq:beta-splitting}, %and $d\eta^i=d\eta^i_0$,
\begin{align}\label{eq-beta-deta}
0 = d\eta^i|_{\mathcal D} \wedge\beta
&= {
d\eta^i|_{\mathcal D}\wedge \beta_0}
+\sum\nolimits_{a=1}^s {\eta^a\wedge\bigl(
d\eta^i|_{\mathcal D}\wedge \mu_a\bigr)}
+\big(\sum\nolimits_{1\le a<b\le s} {f_{ab}\,\eta^a\wedge\eta^b\big)\wedge d\eta^i|_{\mathcal D}}.
\end{align}
Since different $\eta$-degrees lie in a direct sum, each component vanishes separately:

\smallskip
$\bullet$ Zero $\eta$-part (the first term of \eqref{eq-beta-deta}): $d\eta^i|_{\mathcal D}\wedge\beta_0=0$ in $\Omega^4({\mathcal D})$. 
Now, as $d\eta^i|_{\mathcal D}\in\Omega^2({\mathcal D})$ and by the  decomposition \eqref{eq:beta-splitting}, 
$\beta_0 \in \Omega^2({\mathcal D})$ and since 
$\mathcal E_{d\eta^i|_{\mathcal D}}:\Omega^2({\mathcal D})\to\Omega^4({\mathcal D})$ is injective for $n\ge3$ (by \cite[Proposition~1.1]{BPD-2003}), we get $\beta_0=0$.

\smallskip

$\bullet$ {One $\eta$-part}: For each $a$, 
$d\eta^i|_{\mathcal D}\wedge \mu_a=0$ in $\Omega^3({\mathcal D})$. 
Again, since $\mathcal E_{d\eta^i|_{\mathcal D}}: \Omega^1({\mathcal D})\to\Omega^3({\mathcal D})$ is injective 
(in particular, for $n\ge3$), $\mu_a=0$ for all $a$.

\smallskip

$\bullet$ {Two $\eta$-parts}:
From Step~4 we are left with
$\big(\sum\nolimits_{1\le a<b\le s} f_{ab}\,\eta^a\wedge\eta^b\big)
\wedge d\eta^i|_{\mathcal D} = 0$.
Wedge this identity with $(d\eta^i|_{\mathcal D})^{\,n-1}$ to obtain
%\[
$\big(\sum\nolimits_{a<b} f_{ab}\,\eta^a\wedge\eta^b\big) \wedge (d\eta^i|_{\mathcal D})^{\,n} = 0$.
%\]
By Step~1, $(d\eta^i|_{\mathcal D})^{\,n}$ is a nowhere-vanishing horizontal top form on ${\mathcal D}$.
Fix any point $p\in M$.  
Then $\{\,\eta^a_p\wedge\eta^b_p\,\}_{a<b}$ is a basis of
$\Lambda^2(\mathrm{span}\{\eta^1_p,\dots,\eta^s_p\})$, and wedging with
$(d\eta^i|_{\mathcal D})^{\,n}\ne0$ 
%simply 
multiplies these basis elements by the same nonzero
horizontal volume form. 
Hence, $\{\,\eta^a_p\wedge\eta^b_p\,\wedge (d\eta^i|_{\mathcal D})^{n} \}_{a<b}$ are linearly independent $(2n+2)$-forms on $M^{2n+s}$ where $s>1$. Applying this and the equality above forces $f_{ab}(p)=0$ for $a<b$. Since $p\in M$ is arbitrary, $f_{ab}\equiv0$ on $M$. Thus, 
$\beta=0$,
%\beta_0+\sum_a\eta^a\wedge\mu_a+\sum_{a<b}f_{ab}\,\eta^a\wedge\eta^b=0$, 
and $\mathcal E_{d\eta^i}$ is injective on $\Omega^2(M)$ for $n\ge 3$.
\end{proof}
%%%%%%%%%%%%%%%%%%%%%%%%%%%%%

\section*{Conclusion}

We have shown that the weak nearly ${\cal C}$-structure is useful for studying metric $f$-structures, almost contact metric structures, and Killing vector fields.
Some results on nearly cosymplectic manifolds and weak nearly cosymplectic manifolds (see \cite{E-2005,NDY-2018,rov-128}) were generalized to weak nearly ${\cal C}$-manifolds satisfying \eqref{E-nS-10} and \eqref{E-nS-04c}. 
Namely, the splitting Theorem~\ref{Th-4.5} was proven and the $(4+s)$-dimensional case was characterized in Theorem~\ref{T-4.4}.
Their consequences (without the requirement \eqref{E-nS-10} and \eqref{E-nS-04c}) present new results for nearly ${\cal C}$-manifolds.
Based on the numerous applications of the metric $f$-structure, we expect that the weak nearly ${\cal C}$-structure will also be useful for differential geometry 
(e.g., the theory of $\mathfrak{g}$-foliations, $s$-cosymplectic structures and $s$-contact structures) and theoretical~physics.

\section*{Statements and Declarations}

\begin{itemize}
\item\vskip-2.5mm \textbf{Funding}: 
Sourav Nayak is financially supported by a UGC research fellowship (Grant No. 211610029330). Dhriti Sundar Patra would like to thank the Science and Engineering Research Board (SERB), India, for financial support through the Start-up Research Grant (SRG) (Grant No. SRG/2023/002264).

\item\vskip-2.5mm \textbf{Conflict of interest/Competing interests}: The authors have no conflict of interest or financial interests for this article.

\item\vskip-2.5mm \textbf{Ethics approval}: The submitted work is original and has not been submitted to more than one journal for simultaneous consideration.

\item\vskip-2.5mm \textbf{Availability of data and materials}: This manuscript has no associated data.

\item\vskip-2.5mm \textbf{Authors' contributions}: Conceptualization, methodology, investigation, validation, writing-original draft, review, editing, and reading have been performed by all authors of the paper.
\end{itemize}

\end{document}